\documentclass[aip,sd,amsmath,amssymb,
preprint,%
nofootinbib
]{revtex4-1}

\draft 

\usepackage{dcolumn}
\usepackage{bm}
\usepackage{graphicx}
\usepackage{subcaption}
\usepackage[english]{babel}
\usepackage{amsthm}
\usepackage{epstopdf}

\theoremstyle{definition}
\newtheorem{definition}{Definition}[section]
\newtheorem{theorem}{Theorem}[section]

\newtheorem{lemma}[theorem]{Lemma}

\theoremstyle{remark}

\usepackage{color,soul}

\graphicspath{{../figures/}}

\newcommand{\R}{\mathbb{R}}

\newcommand{\N}{\mathcal{N}}
\newcommand{\stab}{\text{stab}}

\begin{document}


\title{Rate-induced tipping from periodic attractors: partial tipping and connecting orbits} 

\author{Hassan M. Alkhayuon}
\email[]{ha317@exeter.ac.uk}
\author{Peter Ashwin}
\email[]{p.ashwin@exeter.ac.uk}

\affiliation{Centre for Systems, Dynamics and Control, Department of Mathematics, University of Exeter, Exeter EX4 4QF, UK}

\date{\today}

\begin{abstract}
We consider how breakdown of the quasistatic approximation for attractors can lead to rate-induced tipping, where a qualitative change in tracking/tipping behaviour of trajectories can be characterised in terms of a critical rate. Associated with rate-induced tipping (where tracking of a branch of quasistatic attractors breaks down) we find a new phenomenon for attractors that are not simply equilibria: partial tipping of the pullback attractor where certain phases of the periodic attractor tip and others track the quasistatic attractor. For a specific model system with a parameter shift between two asymptotically autonomous systems with periodic attractors we characterise thresholds of rate-induced tipping to partial and total tipping. We show these thresholds can be found in terms of certain periodic-to-periodic (PtoP) and periodic-to-equilibrium (PtoE) connections that we determine using Lin's method for an augmented system.
\end{abstract}


\maketitle 

\begin{quotation}
Rate-induced tipping is a mechanism where reaching a critical rate of change (rather than a critical value) of a parameter leads to a sudden change in a system's attracting behaviour \cite{Ashwin2011}. Although there have been several studies of this mechanism for systems with equilibrium attractors, rate-induced tipping from more general attractors (including periodic orbits) is less well understood. We tackle this problem for parameter shift systems \cite{Ashwin2017} by considering properties of forward limits of local pullback attractors, with respect to changes in the rate of the parameter shift. One of the key observations of this paper is that the system may undergo {\em partial tipping} before reaching full tipping: partial tipping occurs when some orbits still track the quasistatic attractor whilst others tip. We also show that the distinction between partial and full tipping can in some circumstances be related to the presence of global connecting orbits in an extended system, and we compute these thresholds using Lin's method. 
\end{quotation}

\section{Introduction}

Motivated by studies of climate \cite{Lenton2011, Schellnhuber2009, Ditlevsen2010},  ecological\cite{Scheffer2008,Scheffer2009}, financial\cite{May2008,Yukalov2009} and biological systems\cite{Nene2010}, the importance of \textit{tipping points} in understanding sudden changes has been a focus of increasing interest in the last few years. Although there is no agreed definition, a tipping point occurs when a system has a sudden, irreversible change in output in response to a small change in input. This change can be associated with a bifurcation (B-tipping), external noise (N-tipping) that can change the stability of multistable system, or with a critical rate (R-tipping) when a system fails to track a continuously changing quasistatic attractor \cite{Scheffer2008,Ashwin2011}. Whilst N- and B-tipping are relatively well studied, rate-induced tipping (R-tipping) has only recently been identified \cite{Wieczorek2011,Ashwin2011} as a distinct mechanism that can cause tipping in a system where there is no bifurcation or noise involved but where the system is {\em nonautonomous} (i.e. not only the solutions but the system itself varies with time).  Since then, a number of papers have studied R-tipping and related effects either using the theory of fast-slow dynamical systems \cite{Kuehn2011,Perryman2014} or notions from nonautonomous stability theory \cite{Hoyer-Leitzel2017,Ashwin2017,Li2016}. In particular, it has been suggested that {\em local pullback attractors} (where typical initial conditions are chosen from some open region in the distant past) provide a suitable setting to describe such transitions\cite{Ashwin2017}. Further studies have attempted to provide early warning indicators for this type of tipping points \cite{Ritchie2016,Ritchie2017,Clements2016}.

\citet{Ashwin2017} propose a framework for R-tipping for nonautonomous systems that limit to different autonomous systems in the past and future. They call these {\em parameter shift systems} and propose that R-tipping is associated with a change in properties of a pullback attractor for the associated nonautonomous system. They relate properties of the pullback attractor to those of the quasistatic system at fixed parameters. Most studies\cite{Ashwin2017,Perryman2014,Ritchie2016} have so far only considered R-tipping from pullback attractors that limit to equilibria: this paper generalizes this framework to include cases where the quasistatic attractor is not necessarily an equilibrium. In doing so we find new phenomenon - the appearance of {\em partial tipping} where the phase of the orbit can influence whether it ``tips'' or not, for some open region in parameter space. For a particular example system \eqref{eq:zLambda} we investigate partial tipping - see Figure~\ref{fig:intro}. We relate different types of tipping and boundaries between them to the presence of periodic-to-periodic (PtoP) or periodic-to-equilibrium (PtoE) connections for an extended system, implementing Lin's method to numerically locate boundaries between types of tipping in this example.

The paper is organised as follows: Section~\ref{sec:parshift} examines backward limits of quite general nonautonomous invariant sets in the setting of parameter shifts with rate dependence, and considers the relation between (local) pullback attractors of the nonautonomous system and attractors for the quasistatic system. Theorem~\ref{thm:pbpastlimit} shows the backward limit of a local pullback attractor limits to an attractor for the past limit system.  Section~\ref{sec:rtipping} uses these local pullback attractors to investigate rate-induced tipping for parameter shifts where the quasistatic attractors may be periodic. We define R-tipping in terms of {\em forward limits} of pullback attractors and in Theorem~\ref{thm:tracking} extend previous results\cite{Ashwin2017} for equilibrium attractors to the case of more general branches of attractors. Section \ref{sec:example} studies a specific example of tipping from a branch of periodic orbits, where we demonstrate the different types of tipping are present. For this example (see Figure~\ref{fig:intro}) we show  that the thresholds of R-tipping can be determined using a numerical implementation of Lin's method for computing connecting orbits. We conclude with a discussion of the results in Section~\ref{sec:discuss}.

\begin{figure}
	\subcaptionbox{} [0.32\linewidth]
	{\includegraphics[scale = 0.33]{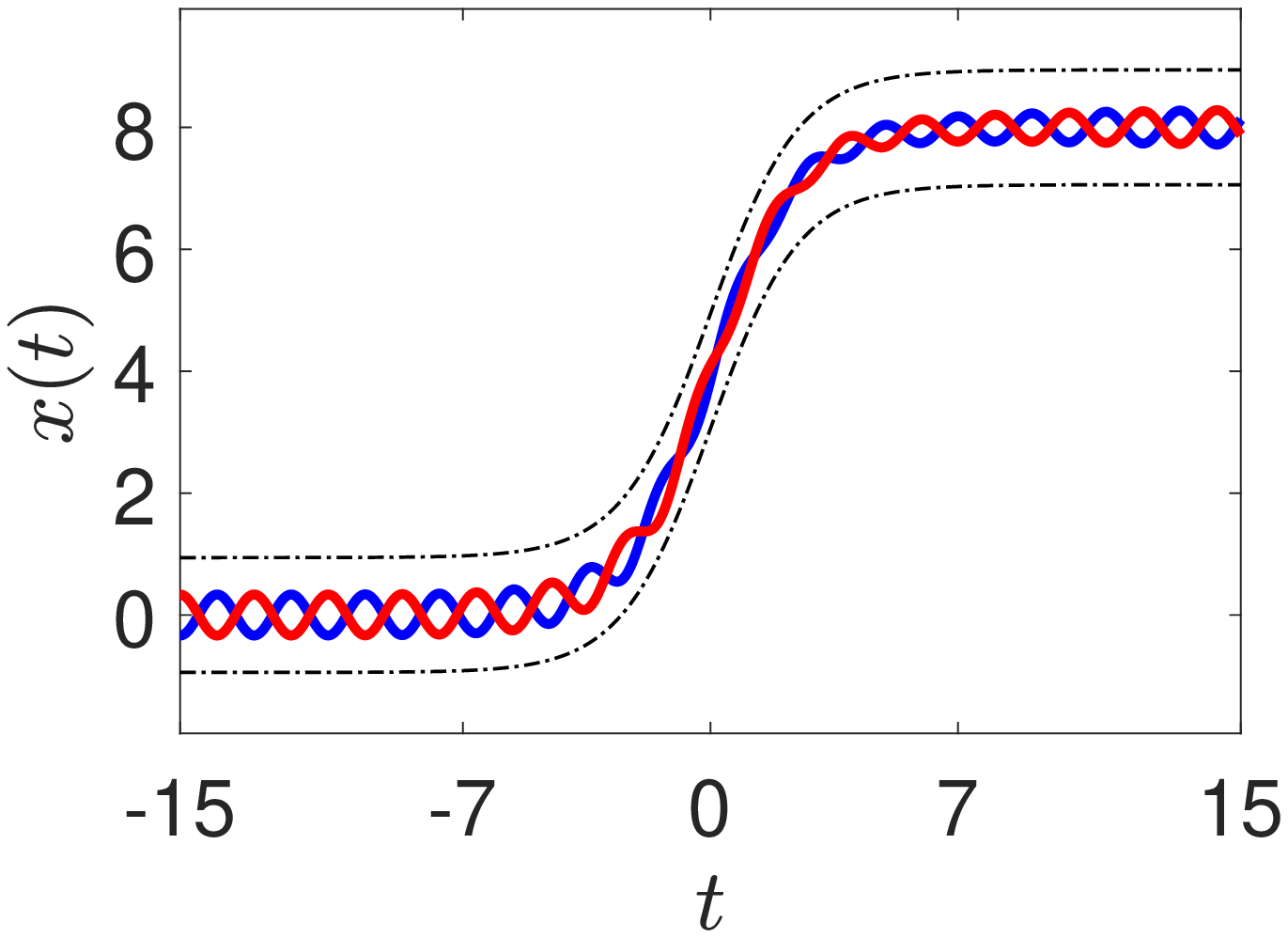}}
	 	\subcaptionbox{} [0.32\linewidth]
	{\includegraphics[scale = 0.33]{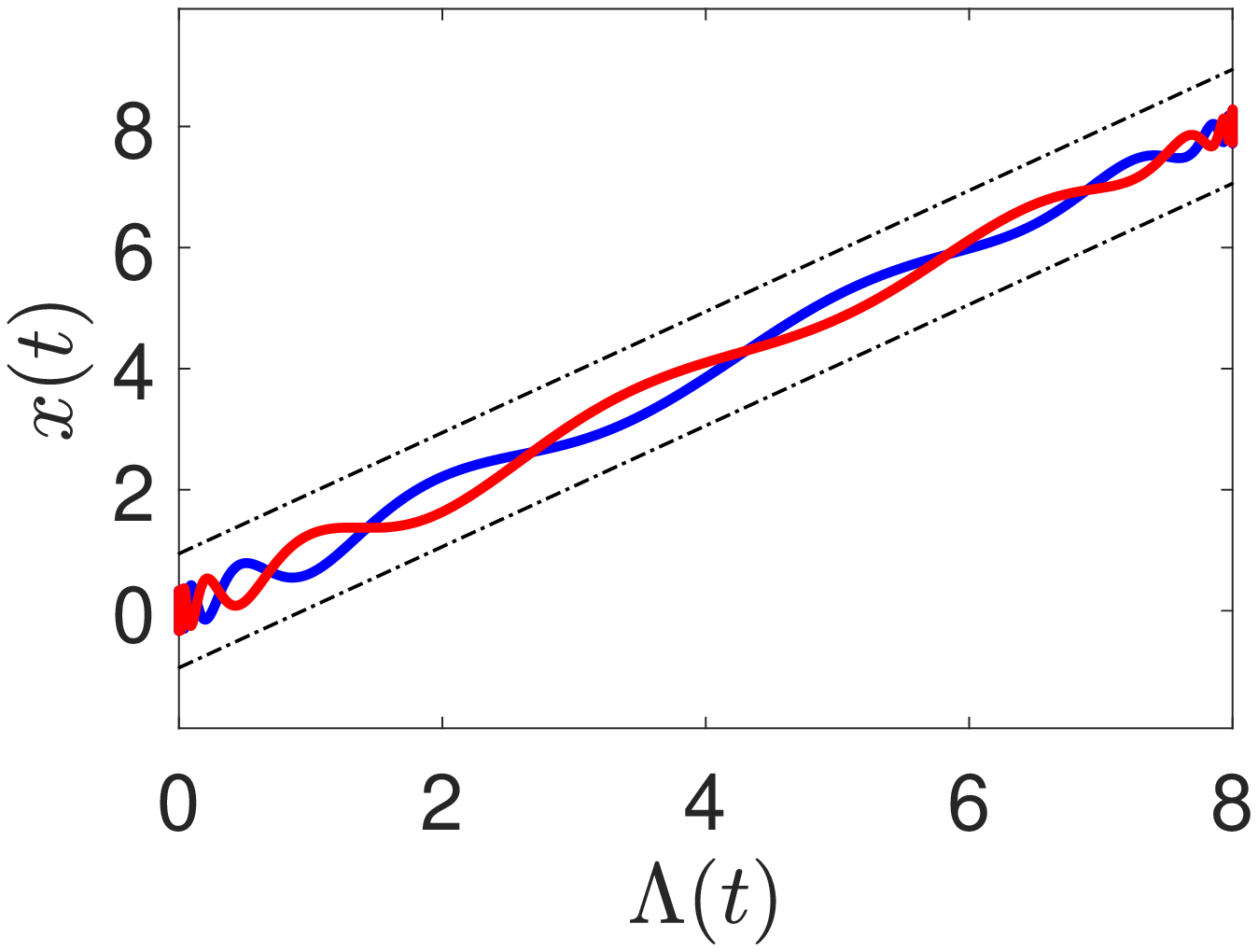}}
	\subcaptionbox{}[0.32\linewidth]
	{\includegraphics[scale = 0.33]{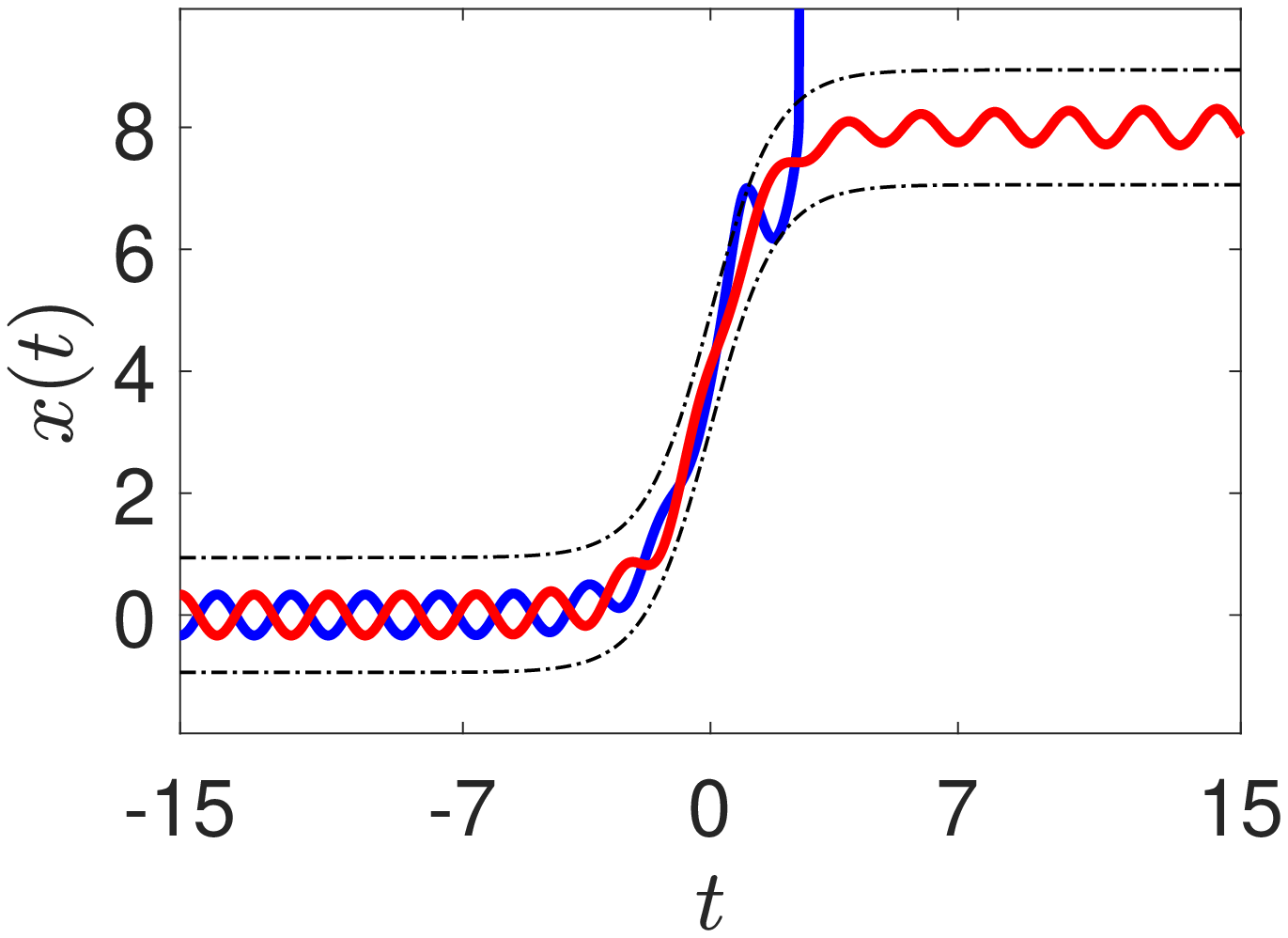}}

	\subcaptionbox{}[0.32\linewidth]
	{\includegraphics[scale = 0.33 ]{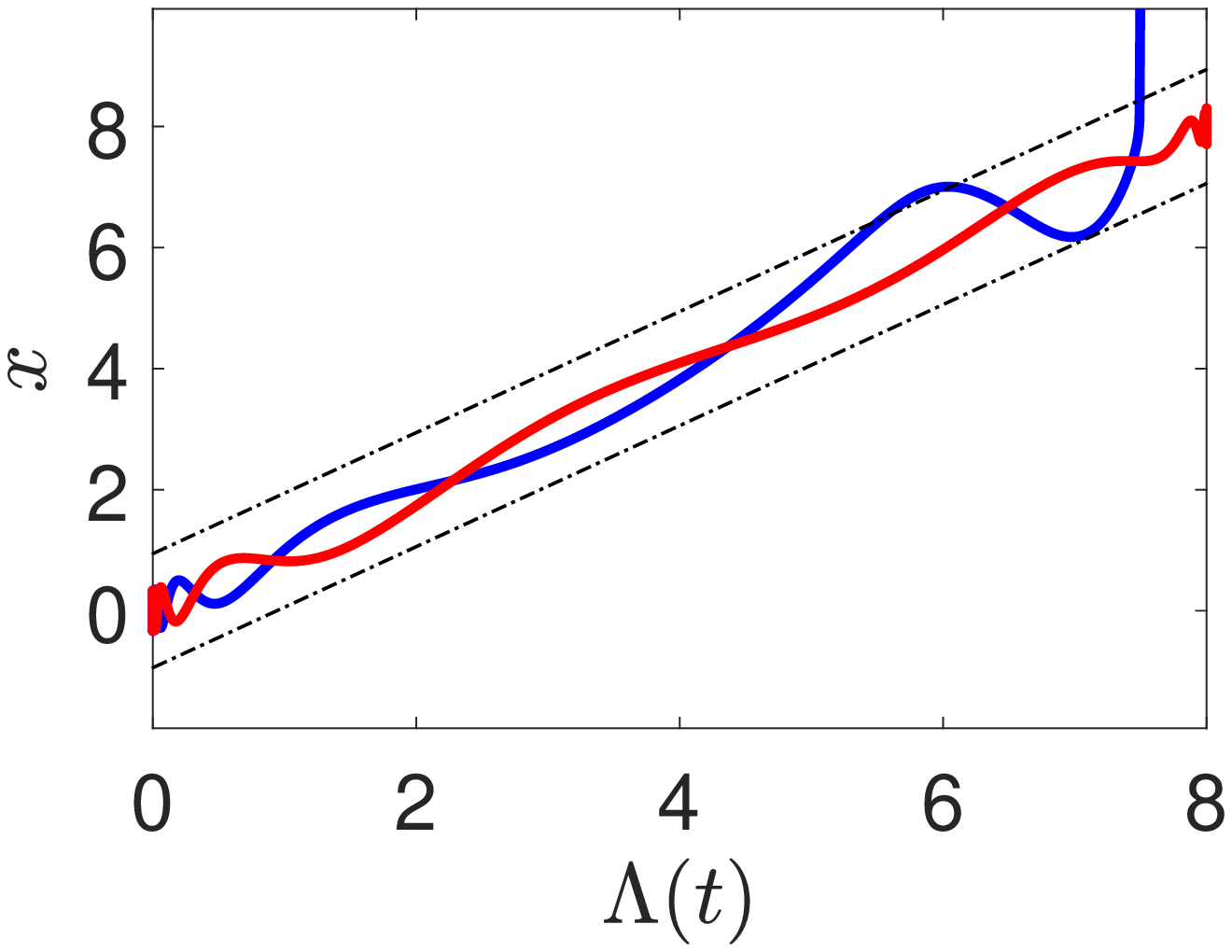}}
	\subcaptionbox{}[0.32\linewidth]
	{\includegraphics[scale = 0.33]{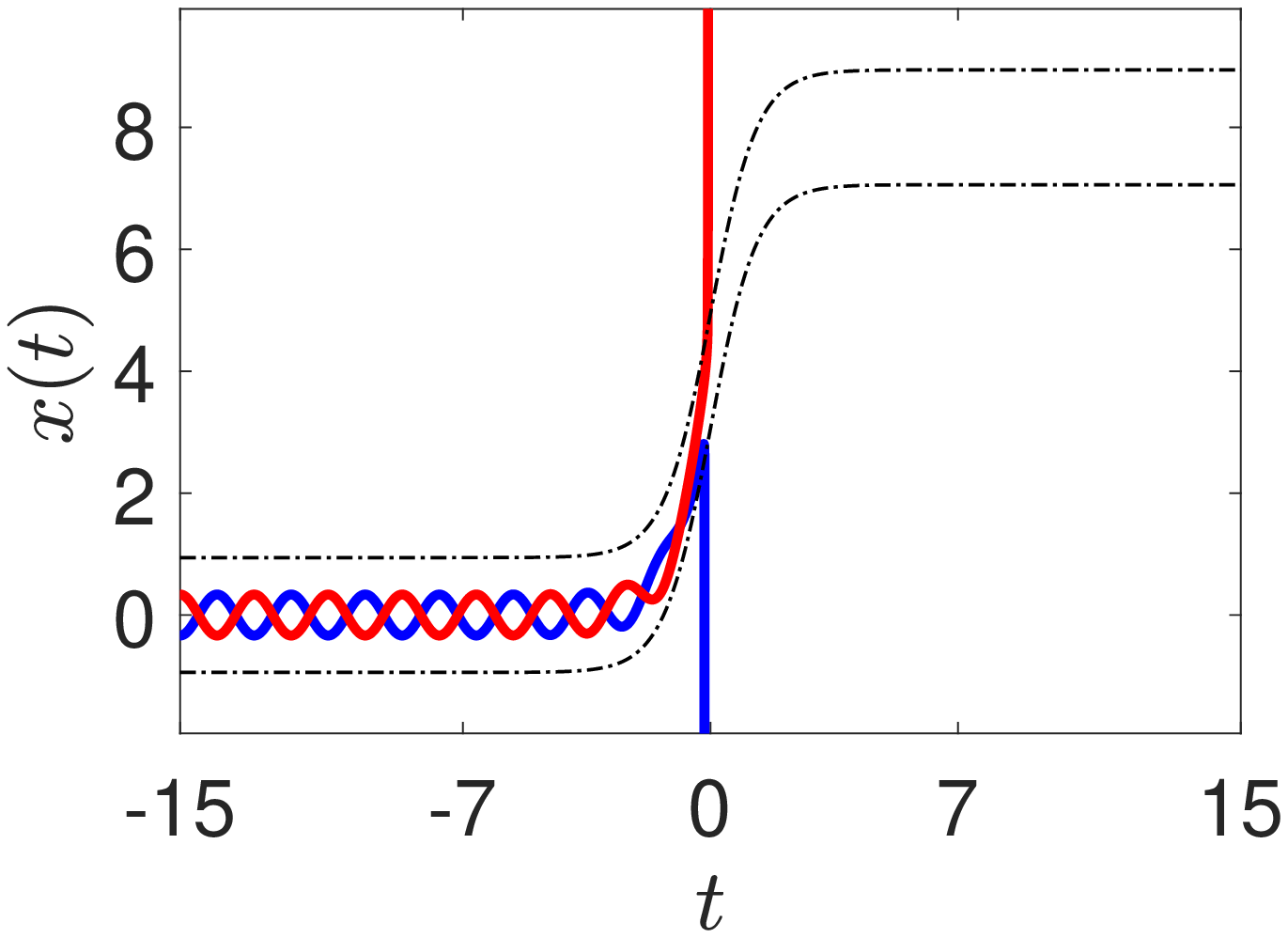}}
	\subcaptionbox{}[0.32\linewidth]
	{\includegraphics[scale = 0.33]{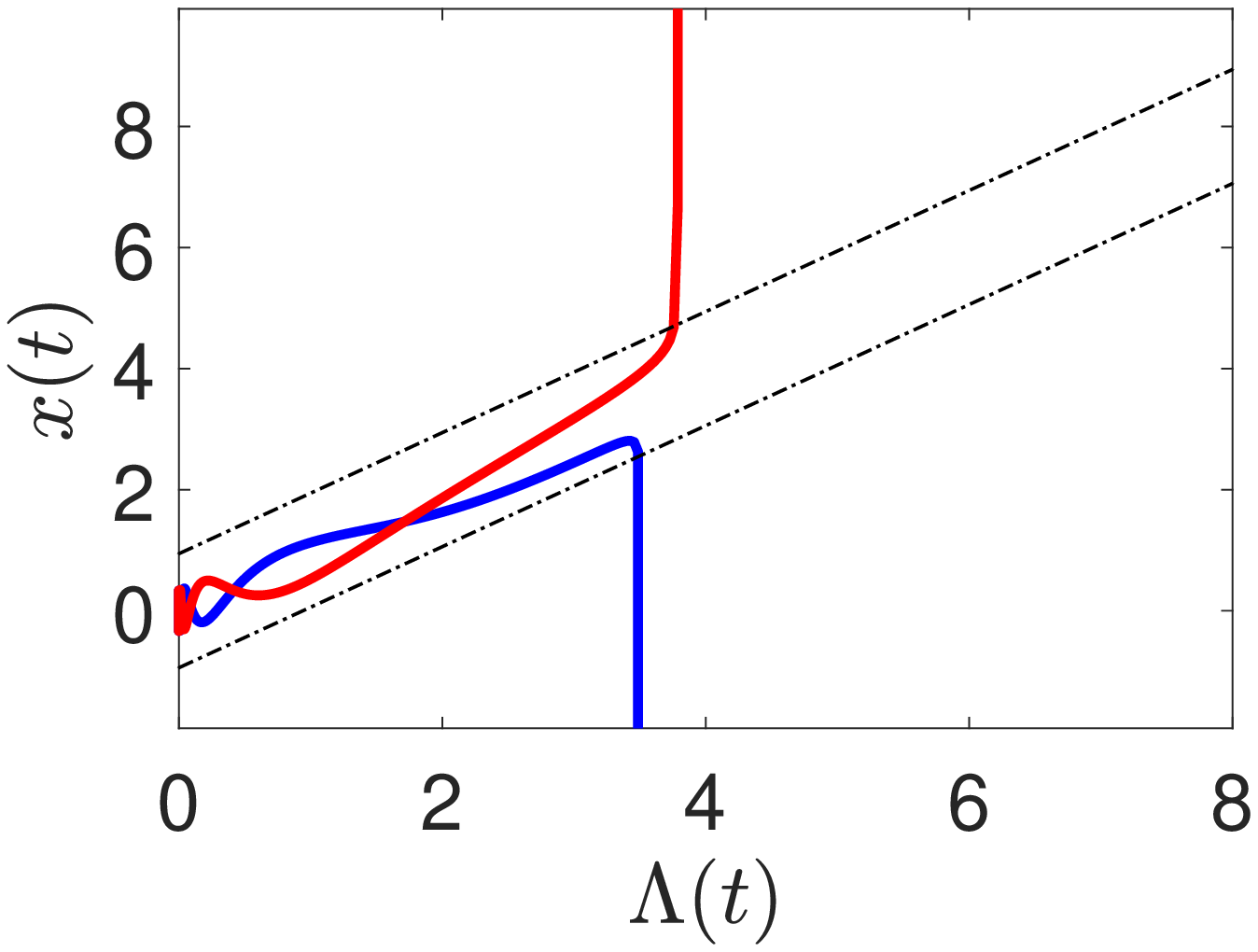}}
	
	\caption{The $x$ components of two typical trajectories (a,c,e) plotted against $t$ and (b,d,f) against $\Lambda$ for system \eqref{eq:zLambda} with $a = 0.1$, $b=1$, $\omega=3$ and $\lambda_{\max}=8$. The black dashed lines show minimum and maximum values of the basin of attraction of a periodic quasistatic attractor. The rates are (a,b) $r=0.1$ showing tracking of the quasistatic attractor,  (c,d) $r=0.1344$ showing evidence of partial tipping (some trajectories escape, some do not) and (e,f) $r=0.2$ showing evidence of total tipping (all trajectories escape).}
	\label{fig:intro}
\end{figure}

\section{Parameter shift systems} 
\label{sec:parshift}

Consider the dynamical system generated by the following nonautonomous differential equation 
\begin{equation}
\dot{x}=f(x,\Lambda(rt)) \label{eq:nonautonomous}
\end{equation}
where $x\in \R^n$, $t,r \in \R$ and $f$ is at least $C^1$ in both arguments. We fix $\lambda_-<\lambda_+$ and call a smooth function a {\em parameter shift}\cite{Ashwin2017} from $\lambda_{-}$ to $\lambda_{+}$ if it varies between these limiting values, more precisely if it is a function $\Lambda: \R \rightarrow (\lambda_-,\lambda_+)$ such that:
\begin{itemize}
\item $\lim_{\tau \rightarrow \pm \infty} \Lambda(\tau) = \lambda_\pm$
\item $\lim_{\tau \rightarrow \pm \infty} d\Lambda/ d\tau = 0$. 
\end{itemize} 
We denote the solution (also called the solution cocycle) of the nonautonomous system \eqref{eq:nonautonomous} with $x(s)=x_0$ by $\Phi(t,s,x_0):= x(t)$ (Note that $\Phi$ depends on $r$ and $\Lambda$ but we will suppress this dependence in most cases). There is an associated autonomous system for \eqref{eq:nonautonomous}, namely 
\begin{equation}
\dot{x}=f(x,\lambda) \label{eq:autonomous}
\end{equation}
where $\lambda$ is constant and denote the solution flow of (\ref{eq:autonomous}) by $\phi_{\lambda}(t,x)$. \citet{Ashwin2017} consider cases where the only attractors of (\ref{eq:autonomous}) are equilibria; we allow the system to have more general attractors. As in \citet{Ashwin2017}, we aim to understand attraction properties of \eqref{eq:nonautonomous} with reference to properties of attractors \eqref{eq:autonomous}. More precisely, we define {\em backward and forward limits of the pullback attractor} of \eqref{eq:nonautonomous} and relate these to attractors for the limiting cases
\begin{equation}
\dot{x}=f(x,\lambda_{\pm}). \label{eq:pastlim}
\end{equation}
We refer to (\ref{eq:pastlim}) in the case $\lambda_{-}$ as the {\em past limit system} and in the case $\lambda_{+}$ as the {\em future limit system}.

\subsection{Bifurcations of the autonomous system}

 Recall that a compact $\phi_\lambda$-invariant subset $M \subset \R^n$ is \textit{asymptotically stable} if it satisfies: 
\begin{itemize}
\item For all $\epsilon > 0$ there exists a $\delta=\delta (\epsilon) > 0$ such that
$$
d(\phi_{\lambda}(t,y),M)< \epsilon \ \text{for all} \ t>0 \ \text{for all} \ y\in\mathcal{N}_{\delta}(M).
$$
\item There exists an $\eta > 0$ such that 
$$
\lim_{t \rightarrow \infty} d(\phi_\lambda (t,y),M)=0 \ \text{for all} \ y \in \mathcal{N}_\eta (M),
$$
\end{itemize}
where $d(X,Y):=\sup_{x\in X}\inf_{y\in Y} \|x-y\|$ is the Hausdorff semi-distance between two non-empty compact subsets $X$ and $Y$ of $\mathbb{R}^n$, the distance from a point $x$ to a set $Y$ is given by $d(x,Y):= d(\{x\},Y)$, and the $\eta$-neighbourhood of $M$ is defined
$$
\mathcal{N}_\eta (M):=\{x\in\R^n~:~d(x,M)<\eta\}.
$$
The Hausdorff distance\cite{Kloeden} between two nonempty compact subsets $X$ and $Y$ of $\mathbb{R}^n$ is defined
$$
d_H \left(X,Y \right) := \max \left\{ d\left(X,Y\right), d \left(Y,X\right)\right\}.
$$
We say a connected compact invariant set $M\subset \R^n$ is an {\em exponentially stable attractor} \footnote{See  \citet[Definition~5.34]{Hartman2014} for globally exponentially stable equilibrium}  for $\phi_{\lambda}$ if there are $\mu>0$, $\eta>0$ and $C\geq 1$ such that

\begin{equation}
\label{eq:linstab}
d(\phi_{\lambda}(t,x),M)\leq  C e^{-\mu t} d(x,M)\ 
\end{equation}
for all $x \in \mathcal{N}_{\eta}(M)$ and $t>0$. Note this implies that $M$ is asymptotically stable.  

Let us denote the set of all exponentially stable attractors by $\mathcal{X}_\text{stab}$: this includes hyperbolic attracting equilibria and periodic orbits: we call $\overline{\mathcal{X}_\text{stab}} \setminus \mathcal{X}_\text{stab}$ the set of {\em bifurcation points}. A continuous set valued function $A(\Lambda(\tau))$, where $A(\Lambda(\tau)) \in \overline{\mathcal{X}_\text{stab}}$, for all $\tau\in \R$, is called a \textit{stable path}. If there exists a choice of $\mu,\eta,C$ (independent of $\lambda$) such that (\ref{eq:linstab}) holds then we say the path is {\em uniformly stable}. A uniformly stable path is called {\em stable branch}. Note that a path can include several stable branches joined at bifurcation points, however in this paper we restrict to stable branches. 

The example in Section~\ref{sec:example} only has branches of attractors of (\ref{eq:autonomous}) that are periodic orbits and equilibria but unless indicated, the remaining results hold for branches of more general attractors. 

\subsection{Local pullback attractors and backward limits}

We recall some concepts from the nonautonomous (set valued) theory of dynamical systems\cite{Kloeden}. A set-valued function of $t\in\R$ (family of nonempty subsets of $\mathbb{R}^n$)  is called a {\em nonautonomous set} and written $\mathcal{A}=\{A_t\}_{t\in\R}$ with $A_t\subset \R^n$ the {\em fibre}. We use the {\em upper limit} of a sequence of sets\cite{set-valued} to define the limiting behaviour of $\mathcal{A}$. Note there is also a lower limit\cite{set-valued, Rasmussen2008}, but the upper limit captures the asymptotic behaviour in a maximal sense.

For a nonautonomous set $\mathcal{A}=\{ A_t \}_{t \in \R}$ the {\em upper forward limit} $A_{+\infty}$ and the {\em upper backward limit} $A_{-\infty}$ are defined as: 
$$
A_{+\infty} := \limsup_{t\rightarrow\infty}A_t=\bigcap_{\tau > 0} \overline{\bigcup_{t\geq\tau} A_t}
$$
$$
A_{-\infty} := \limsup_{t\rightarrow-\infty}A_t=\bigcap_{\tau > 0} \overline{\bigcup_{t\leq -\tau} A_t}.
$$
A nonautonomous set $\mathcal{A}=\{ A_t \}_{t \in \R}$ with $A_t\subset \R^n$ is called {\em invariant} for \eqref{eq:nonautonomous} if $\Phi(t,s,A_s)=A_t$ for all $t,s$. A nonautonomous set is called {\em compact, bounded} etc if $A_t$ is compact, bounded etc for all $t\in\R$.

Note that for general nonautonomous systems, $A_{\pm \infty}$ may be at least as complex as an invariant set for an autonomous system (e.g. it may have fractional dimension, or indeed empty). However, for the parameter shift systems that we consider those limits that can be linked to the behaviour of past and future limit systems as follows. The first result shows that if there are past (future) limit systems then the backward limit $A_{-\infty}$ (forward limit $A_{+\infty}$) is invariant for the limit system that we define to be $\phi_{\pm}:=\phi_{\lambda_{\pm}}$.

\begin{lemma}
For a parameter shift from $\lambda_-$ to $\lambda_+$ and a nonautonomous invariant set $\mathcal{A}$ with fibre $A_t$, if $A_{\pm\infty}$ is bounded then we have
$$
\phi_{\pm}(s,A_{\pm\infty})=A_{\pm\infty}
$$
for all $s$.
\label{lem:invariant limit}
\end{lemma}

\begin{proof}
We prove in detail for the past limit case: the future limit proof follows similarly. Let us denote $U_{\tau}:=\bigcup_{t\leq \tau} A_t$ so that $A_{-\infty}=\bigcap_{\tau<0}\overline{U_{\tau}}$. Note that 
$$
U_{\tau'}\subset U_{\tau}~~\mbox{ and }~~~\Phi(\tau,\tau',U_{\tau'})=U_{\tau}
$$
for any $\tau'<\tau$, where the first containment follows from the definition of $U$, and the second from the invariance of $A_t$ under the cocycle. In particular, the second statement can be written
$$
d_H( \Phi(\tau',\tau'-T,U_{\tau'-T}),U_{\tau'})=0
$$
for any $\tau'$ and $T>0$.
Pick any compact and convex set $K$ that contains a neighbourhood of $A_{-\infty}$.  Applying Lemma 5.1(i) of \citet{Rasmussen2007}, means that for any $T>0$ and $\epsilon>0$ there is a $\tau_0<-T$ such that
$$
\|\Phi(\tau'+t,\tau',x) - \phi_{-}(t,x)\|<\epsilon
$$
for all $\tau'<\tau_0$, $0\leq t\leq T$ and $x\in K$. 

Pick a sufficiently negative $\tilde{\tau}<0$ that $U_{\tilde{\tau}}\subset K$ and fix any $T>0$. For every $\epsilon>0$ there is an $\tau_0(\epsilon)<\min(-T,\tilde{\tau})$ such that 
$$
\|\Phi(\tau',\tau'-T,x) - \phi_{-}(T,x)\|<\epsilon
$$
for all $\tau'<\tau_0(\epsilon)$ and $x\in U_{\tau'-T}$. This implies that
$$
d_H(\phi_-(T,U_{\tau'-T}),\Phi(\tau',\tau'-T,U_{\tau'-T}))<2\epsilon
$$
for any $\tau'<\tau_0(\epsilon)$. Applying the triangle inequality
\begin{eqnarray*}
d_H(\phi_-(T,U_{\tau'-T}),U_{\tau'}) &\leq &
d_H(\phi_-(T,U_{\tau'-T}),\Phi(\tau',\tau'-T,U_{\tau'-T}))+d_H(\Phi(\tau',\tau'-T,U_{\tau'-T}),U_{\tau'})
\end{eqnarray*}
implies that for all $\tau'<\tau_0(\epsilon)$ we have
$$
d_H(\phi_-(T,U_{\tau'-T}),U_{\tau'}) \leq 2\epsilon
$$
in particular for $x\in A_{-\infty}$ and fixed $T>0$ implies $\phi(T,x)\in A_{-\infty}$. Allowing $T$ to vary gives the proof for all $T>0$: note that $\phi(T,.)$ is a diffeomorphism and hence the result holds for all $T$.
\end{proof}

The next definition generalizes Definition~2.3 in \citet{Ashwin2017}.

\begin{definition} \label{Def.pb} 
Suppose that $\mathcal{A}=\{A_t\}_{t\in \R}$ is a compact $\Phi$-invariant nonautonomous set. We say $\mathcal{A}$ is a \textit{(local) pullback attractor} that {\em attracts }$U$ if there exists a bounded open set $U$ containing the upper backward limit of $\mathcal{A}$ that satisfies  
$$
\lim_{s\rightarrow -\infty} d(\Phi(t,s,U),A_t)=0
$$
for all $t \in \R$.
\end{definition}

The following result generalizes Theorem 2.2 in \citet{Ashwin2017} - it gives a sufficient condition that there is a local pullback attractor whose backward limit is contained within an attractor of the past limit system.

\begin{theorem} \label{thm:pbpastlimit}
Suppose that $A_-$ is an asymptotically stable attractor for the past limit system $\phi_-$. Then there is local pullback attractor of  \eqref{eq:nonautonomous} whose (upper) backward limit  is contained in $A_-$.
\end{theorem}  

We delay the proof of Theorem~\ref{thm:pbpastlimit} to give two lemmas that will be used in the proof.

\begin{lemma}
\label{lem:absorbing}
Assume that  $A_-$ is an asymptotically stable attractor for the past limit system $\phi_-$. Then there is $\tilde{\eta}>0$ such that for all $\eta \in (0,\tilde{\eta}]$ and all $\delta>0$ there exist $\tau>0$ and $\tilde{\tau}>0$ such that $\Phi(t,s,\mathcal{N}_\eta(A_-)) \subset \mathcal{N}_\delta(A_-) $, for all $t$ and $s$ such that $s<t-\tilde{\tau}$ and $t<-\tau$.
\end{lemma}

\begin{proof}
Asymptotic stability of $A_-$ means that there is a $\tilde{\eta}>0$ such that for any $0 < \eta<\tilde{\eta}$ we have
$$
\lim_{s\rightarrow\infty} d(\phi_-(s,\N_{\eta}(A_-)),A_-) = 0.
$$ 
This means that for any $\delta > 0$ there is $\tilde{\tau}>0$ such that  $d(\phi_-(k,\mathcal{N}_{\eta}(A_-)),A_-)< \delta/2$ for all $k>\tilde{\tau}$. By \citet{Rasmussen2008} Lemma 5.1, for any $\delta>0$ and $k>\tilde{\tau}$ there is $\tau>0$ such that
$$
d_H(\Phi(u,u-k,\mathcal{N}_{\eta}(A_-)),\phi_-(k,\mathcal{N}_{\eta}(A_-)))< \delta/2
$$  
for all $u<-\tau$. The triangle inequality of Hausdorff semi-distance implies
\begin{eqnarray*}
d\left(\Phi(u,u-k,\mathcal{N}_\eta(A_-)),A_-\right) &\leq& d\left(\Phi(u,u-k,\mathcal{N}_\eta(A_-)), \phi_-(k,\mathcal{N}_\eta(A_-)) \right) \\
&+& d\left(\phi_-(k,\mathcal{N}_\eta(A_-), A_-\right) < \delta /2 + \delta /2 = \delta.
\end{eqnarray*}

for all $u$ and $k$ such that $u<-\tau$ and $u-k<u-\tilde{\tau}$, which completes the proof. 
\end{proof}

We define  $\mathcal{A}^{[\Lambda,r,A_-]}:=\{A_t^{[\Lambda,r,A_-]}\}_{t\in\R}$ where:
\begin{equation} \label{eq:pbattractor}
A_t^{[\Lambda,r,A_-]}:=\bigcap_{\tau>0} \overline{\bigcup_{s\leq -\tau} \Phi\left(t,s,\mathcal{N}_{\eta}(A_-)\right)} 
\end{equation}
for all $t\in\R$ (recall that $\Phi$ is the solution of (\ref{eq:nonautonomous}) and so depends on $r$ and $\Lambda$).

\begin{lemma} 
Assume that $A_-$ is asymptotically stable attractor for the past limit system $\phi_{\lambda_-}$. Then the nonautonomous set \eqref{eq:pbattractor} is independent of $\eta$ for all $ \eta \in (0, \tilde{\eta}]$. 
\label{lem:uniqueness} 
\end{lemma}

\begin{proof}
Consider any $\eta$ and $\eta'$ in $(0,\tilde{\eta}]$ and assume $\eta' < \eta$ {\em w.l.o.g} and define
\begin{align*}
A_t &= \bigcap_{\tau>0} \overline{\bigcup_{s\leq -\tau} \Phi\left(t,s,\mathcal{N}_{{\eta}}(A_-)\right)},\\
A_t' &= \bigcap_{\tau>0} \overline{\bigcup_{s\leq -\tau} \Phi\left(t,s,\mathcal{N}_{\eta'}(A_-)\right)}.
\end{align*}
Since $\mathcal{N}_{\eta'}(A_-) \subset \mathcal{N}_{\eta}(A_-)$ we have  
$$
\bigcap_{\tau>0} \overline{\bigcup_{s\leq -\tau} \Phi\left(t,s,\mathcal{N}_{\eta'}(A_-)\right)} \subset \bigcap_{\tau>0} \overline{\bigcup_{s\leq -\tau} \Phi\left(t,s,\mathcal{N}_{\eta}(A_-)\right)}
$$
which means $A_t' \subset {A}_t$. We also have to show that $A_t \subset A_t'$. By Lemma~\ref{lem:absorbing} there exist $\tau$, $\tilde{\tau} > 0$, such that  $\Phi(k,s,\mathcal{N}_\eta(A_-)) \subset \mathcal{N}_{\eta'}(A_-) $, for all $s<k-\tilde{\tau}$ and $k<-\tau$.

Now, for all $t\in\mathbb{R}$  
\begin{eqnarray*}
\Phi\left(t,k,\Phi\left(k,s,\mathcal{N}_\eta(A_-)\right)\right) &\subset & \Phi\left(t,k,\mathcal{N}_{\eta'}(A_-)\right),\\
\Phi\left(t,s,\mathcal{N}_\eta(A_-)\right) &\subset & \Phi\left(t,k,\mathcal{N}_{\eta'}(A_-)\right),\\
\bigcap_{\tau>0} \overline{\bigcup_{s<-\tau-\tilde{\tau}}\Phi\left(t,s,\mathcal{N}_\eta(A_-)\right)} &\subset & \bigcap_{\tau>0} \overline{\bigcup_{k<-\tau}\Phi\left(t,k,\mathcal{N}_{\eta'}(A_-)\right)},\\
A_t &\subset& A_t'.
\end{eqnarray*}
Therefore $A_t=A_t'$ for all $t \in \mathbb{R}$ and so  $\mathcal{A}^{[\Lambda,r,A_-]}$ is independent of choice of $\tilde{\eta}>\eta>0$.
\end{proof}

\begin{proof}{\em (For Theorem~\ref{thm:pbpastlimit})} 
See Appendix~\ref{AppProofTh2.2} for a detailed proof.
\end{proof}

By Lemma~\ref{lem:invariant limit}, $A_{-\infty}^{[\Lambda,r,A_-]}$ is invariant for the past limit system, if $A_-$ is minimal (for example, if it is an equilibrium or periodic orbit) then $A_{-\infty}^{[\Lambda,r,A_-]}= A_-$. We believe that  $A_{-\infty}^{[\Lambda,r,A_-]}= A_-$ in more general cases but are not clear whether additional hypotheses are needed to prove this. However, we note that, as pointed out by an anonymous referee, the proof of Theorem~\ref{thm:pbpastlimit} can be obtained by adapting \citet[Theorem 2.35 and Corollary 2.36]{Rasmussen2007} to this setting. 

\section{Tracking and rate-induced tipping of pullback attractors}
\label{sec:rtipping}

Theorem~\ref{thm:pbpastlimit} highlights that the backward limit of a pullback attractor for the parameter shift system (\ref{eq:nonautonomous}) is related to an attractor of the past limit system. Whether the forward limit of the pullback attractors is related to an attractor of the future limit system, is a more subtle question that depends on choice of rate $r>0$:  

\begin{definition} \label{Def:tipping}
Suppose that $(A(\lambda),\lambda)\subset \mathcal{X}_{stab}$ is a branch of attractors that are exponentially stable for $\lambda\in[\lambda_-,\lambda_+]$. Define $A_{\pm}:=A(\lambda_{\pm})$ and consider the pullback attractor $\mathcal{A}^{[r,\Lambda,A_-]}$ with past limit $A_-=A(\lambda_-)$.
\begin{itemize}
\item We say there is (end-point) {\em tracking} for the system \eqref{eq:nonautonomous} from $A_-$ for some $\Lambda$ and $r>0$ if 
$$
A_{+\infty}^{[\Lambda,r,A_-]} \subset A_+.
$$

\item We say there is {\em partial tipping} if
$$
(A_+)^c \cap A_{+\infty}^{[\Lambda,r,A_-]}\neq \emptyset,~\mbox{ and }~A_+\cap A_{+\infty}^{[\Lambda,r,A_-]}\neq  \emptyset. 
$$
\item We say there is {\em total tipping} if
$$
A_+\cap A_{+\infty}^{[\Lambda,r,A_-]}= \emptyset.
$$ 
\item We say there is {\em tipping} for the system if there is partial or total tipping, i.e. if
$$
A_{+\infty}^{[\Lambda,r,A_-]} \not\subset A_+.
$$
\item
For a given $A_-$ and $\Lambda$ there will be partition of the positive half axis into disjoint subsets where there is tracking, partial tipping or total tipping. If $r_c$ is in the closure of two of these sets we say it is a {\em critical rate} or {\em threshold} for {\em rate-induced tipping.}
\item It is possible to have an isolated value of the rate $r_0$ that gives partial tipping but that separates two subsets of $r>0$ where the system has end-point tracking. In this case we say the system has {\em invisible tipping}.      
\end{itemize}
\end{definition}

By analogy with \citet[Theorem 2.4]{Ashwin2017} we expect for sufficiently small $r>0$ that the pullback attractor will track (i.e. remain close to) the branch $A(\lambda)$.  This is expressed more precisely in the following result.

\begin{theorem} \label{thm:tracking}
Suppose that $(A(\lambda),\lambda)\subset \mathcal{X}_{\stab}$ is a branch of attractors that is uniformly stable for $\lambda\in[\lambda_-,\lambda_+]$ and suppose $\Lambda$ is a parameter shift. Define $A_{\pm}=A(\lambda_{\pm})$ and the pullback attractor $\mathcal{A}^{[r,\Lambda,A_-]}$ with fibres $A_t^{[r,\Lambda,A_-]}$ as in (\ref{eq:pbattractor}). Then for all $\epsilon>0$ there exists a $\delta>0$ such that
$$
d\left(A^{[\Lambda,r,A_-]}_t,A(\Lambda(rt)\right)<\epsilon
$$
for all $0<r<\delta$ and $t\in\R$. Moreover, there is a $\delta>0$ such that there is  tracking for all $0<r<\delta$.
\end{theorem}

\begin{proof}
Since $A(\lambda)$ is uniformly stable for all $\lambda\in \left[\lambda_-,\lambda_+\right]$ then there exist $\mu>0$, $\eta>0$ and $C\geq 1$ (which we fix from hereon in the proof) such that 
\begin{equation}
d\left( \phi_\lambda(t,x),A(\lambda)  \right) < C e^{-\mu t} d\left( x, A(\lambda)\right)
\label{eq:linstab2}
\end{equation}
for all $x\in\mathcal{N}_\eta \left(A(\lambda)\right)$ and $t>0$. 

Pick any $0<\epsilon<\eta$, consider any $t\in\R$ and $\lambda=\Lambda(rt)$. By (\ref{eq:linstab2}) $d\left(\phi_{\lambda}\left(s,x \right),A(\lambda)\right)<\epsilon e^{-\mu s}/3$, for all $x\in\mathcal{N}_{\epsilon/C}\left(A(\Lambda(rt)) \right)$ and $s>0$. In particular, we can pick $s>0$ independent of $t$ such that $e^{-\mu s}=1/C$ and so 
$$
\phi_{\lambda}\left(s, \mathcal{N}_{\epsilon/C}\left(A(\lambda) \right)\right)\subset \mathcal{N}_{\epsilon/(3C)}\left(A(\lambda) \right).
$$

By the continuity of $\Phi$, for all $s>0$ and $t\in\R$ there exits $\delta_1>0$ such that for all $0<r<\delta_1$ and $x\in \mathcal{N}_{\epsilon}(A(\lambda))$
$$
\|\phi_{\lambda}(s,x) - \Phi(t+s,t,x) \|<\epsilon/(3C).
$$
Again by the continuity of $A(\lambda)$ there exist $\delta_2 >0 $ such that for all $t\in\mathbb{R}$ and $0<r<\delta_2$ 
$$
d_H\left(A(\Lambda(r(t+s))),A(\lambda)\right)<\epsilon/(3C).
$$
Now set $\delta=\min\left\{\delta_1,\delta_2\right\}$, then for all $x\in\mathcal{N}_{\epsilon/C}\left(A(\lambda)\right)$, $t\in\R$ and $0<r<\delta$, 
\begin{align*}
&d\left( \Phi\left( t+s,t,x \right) , A\left(\Lambda \left(r(t+s) \right)
\right)  \right)\\
<&d\left( \Phi(t+s,t,x), \phi_{\lambda}(s,x) \right) + d\left( \phi_{\lambda} (s,x), A(\lambda) \right) \\
& +d\left(  A(\lambda),A(\Lambda(r(t+s))\right) \\
\leq& \| \phi_\lambda(s,x) - \Phi(t+s,t,x) \| + d\left(\phi_\lambda(s,x),A(\lambda) \right)\\
&+ d_H\left(A\left(\Lambda(r(t+s))\right),A(\lambda)\right) \\ 
<& \epsilon/(3C)+\epsilon/(3C)+\epsilon/(3C) = \epsilon/C
\end{align*}
which follows from the triangle inequality for Hausdorff semi-distance, $d(u,v)=\|u-v\|$ for all $u,v \in\mathbb{R}^n$ and $d\left(A,B\right)\leq d_H\left(A,B\right)$ for all $A$, $B$ compact subsets of $\mathbb{R}^n$. This means that for all $0<r<\delta$ and $t\in\R$ there is an $s>0$ such that
$$
\Phi(t+s,t,\mathcal{N}_{\epsilon/C}(A(\Lambda(rt)))\subset \mathcal{N}_{\epsilon/C}(A(\Lambda(r(t+s))).
$$
By Theorem \ref{thm:pbpastlimit}, $A_{-\infty}^{[\Lambda,r,A_-]}\subset A_-$ which means for all $\epsilon >0$ there is an $\tau>0$ such that 
$d( A_t^{[\Lambda,r,A_-]},A(\Lambda(rt)) )< \epsilon/C $ for all $t<-\tau$. Therefore, for all $0<\epsilon<\eta$ there exists a $\delta >0$ such that for all $0<r<\delta$ and $t\in \R$ we have
$$
d\left( A_t^{[\Lambda,r,A_-]}, A(\Lambda(rt)) \right)<\epsilon/3C.
$$ 
To prove the second part of the theorem, we define
$$
C_\tau = \overline{\bigcup_{s>\tau} A_s^{[\Lambda,r,A_-]}}.
$$
Note that $C_{u}\subset C_{\tau}$ for any $u>\tau$ and $A_{+\infty}^{[\Lambda,r,A_-]}=\bigcap_{\tau>0}C_{\tau}$. Moreover we have $
d_H(C_{\tau},A_{+\infty}^{[\Lambda,r,A_-]})\rightarrow 0 
$ as $\tau\rightarrow \infty$. 

From before, for any $\epsilon >0$ and $t\in \R$ there is $\delta >0$ such that 
$$
d(A_t^{[\Lambda,r,A_-]}, A(\Lambda(rt))) < \epsilon/2C
$$
for all $0<r<\delta$. 

Now from the fact that $d_H (A(\Lambda(rt)),A_+)\rightarrow 0$ as $t\rightarrow \infty$ and the definition of $C_\tau$, we have $ d(C_\tau, A_+) \rightarrow 0$ as $\tau \rightarrow \infty$. Hence, by the triangle inequality of Hausdorff semi-distance
$$
d(A_{+\infty}^{[\Lambda,r,A_-]}, A_+)=0.
$$
Which finishes the proof.
\end{proof}

Although Theorem~\ref{thm:tracking} means that a pullback attractor will track a branch of ``sufficiently stable'' attractors for the nonautonomous system for small enough rates, there is no guarantee this holds for larger rates. Rate-induced tipping occurs precisely when tracking fails to occur. 

\section{An example with partial and total rate-induced tipping}
\label{sec:example}

In this section we consider an example where there is a branch of periodic attractors, and find cases of partial and total tipping. More precisely, consider the following (nonautonomous) system: 

\begin{equation}
\label{eq:pb_example}
\dot{z} = F(z-\Lambda(rt))
\end{equation}

where $z=x+iy \in \mathbb{C}$, the parameter shift $\Lambda(\tau)= \lambda_{\max} \left(\tanh(\tau \lambda_{\max} /2 )+1 \right)/2$ limits to $0$ in the past and $\lambda_{\max}$ in the future, and $F(z)$ is defined by
\begin{equation}
F(z)=(a+i\omega)z -b|z|^2 z + |z|^4 z
\label{eq:F}
\end{equation}
for $a,b,\omega,r$ and $\lambda_{\max} \in \R$, $r,\lambda_{\max} >0$: we set $b=1$ in what follows. Note that $\dot{z}=F(z)$ can be thought of a normal form for a Bautin bifurcation, where a Hopf bifurcation changes criticality at $b=0$. One can view the system autonomously as:

\begin{equation}
\label{eq:zLambda}
\left.\begin{array}{rcl}
\dot{z}		 &=&F(z-\Lambda)\\
\dot{\Lambda}&=&r\Lambda(\lambda_{\max}-\Lambda)
\end{array}\right\}
\end{equation}

 Previous works \cite{Ashwin2011,Perryman2015} has used parameter shift of a subcritical Hopf normal form to investigate rate-induced tipping. Figure~\ref{fig:intro} illustrates numerically that the dynamics of this system may show tracking, and both partial or total tipping from a branch of periodic orbits.

For $r=0$ and any fixed $\Lambda$ there are bifurcation points at $a=0$ and $a=0.25$, that are Hopf and saddle-node bifurcations of periodic orbits respectively. For $0<a<0.25$ the system has an unstable equilibrium point $Z(\lambda):=\lambda+0i$, as well as both stable and unstable periodic orbit.
We denote the radius of the unstable periodic orbit by $R_u:=(1+\sqrt{1-4a})/2$ and the radius of the stable periodic orbit by $R_s:=(1-\sqrt{1-4a})/2$. Note that the stable periodic orbit is
$\Gamma^s(\lambda):=\left\{ \|z-\lambda\|^2=R_s \right\} $ and the unstable periodic orbit is $\Gamma^u(\lambda):=\left\{ \|z-\lambda\|^2=R_u \right\}$. 

For a solution of \eqref{eq:zLambda} and $r>0$ there are two stationary values of $\Lambda$: $\lambda_-=0$ and $\lambda_+=\lambda_{\max}$. Hence in general there are six invariant sets associated with those two limiting values, and we denote them by $Z_-,\Gamma^s_-$ and $\Gamma_-^u$ associated with $\Lambda=\lambda_-=0$ and  $Z_+,\Gamma^s_+$ and $\Gamma_+^u$ associated with $\Lambda=\lambda_+=\lambda_{\max}$. Theorem~\ref{thm:tracking} implies that the upper forward limit of the pullback attractor $\mathcal{A}^{[\Lambda,r,\Gamma^u_-]}$ is the attracting periodic orbit $\Gamma_+^s$ of the future limit systems, for all small enough $r$. However, there can be up to three critical rates of $r$ for all fixed values of the parameters $a, \omega, \lambda_{\max}$ that can give partial, total and even invisible tipping. 

\subsection{Pullback attractors, tipping, and invariant manifolds}

Writing $W^u(X)$ to denote the unstable and $W^s(X)$ the stable manifold of the hyperbolic invariant set $X$. Moreover, we denote the tangent space of $W^{s,u}(X)$ at the point $p \in W^{s,u}(X) $ by $T_pW^{s,u}(X)$. Note that, $W^s(\Gamma^u_+)$ forms the basin boundary of $\Gamma^s_+$, and the branch of stable periodic orbits is uniformly stable.

The various cases of tracking and tipping can be understood in terms of the unstable manifolds of these invariant sets\cite{Aulbach2006}. More precisely, the pullback attractor of \eqref{eq:pb_example} consists of sections of $W^u(\Gamma^s_-)$ for \eqref{eq:zLambda} and we can classify the tracking/tipping as follows:
\begin{itemize}
\item 
If $\Gamma^s_+\subset W^u(\Gamma^s_-)$ then there is end-point tracking of the branch of periodic solutions $\Gamma^s(\Lambda(rt))$.
\item 
If $[\Gamma^s_+]^c \cap W^u(\Gamma^s_-)\neq \emptyset$  then there is tipping: if in addition $\Gamma^s_+\cap W^u(\Gamma^s_-) = \emptyset $ then there is total tipping for this $r$, otherwise it is partial tipping.
\item
This means that, if there is total tipping or tracking then
$$
W^u(\Gamma^s_-)\cap W^s(\Gamma^u_+)=\emptyset.
$$
while if
$$
W^u(\Gamma^s_-)\cap W^s(\Gamma^u_+)\neq \emptyset.
$$
and the intersection is transverse then there is partial tipping.
\item
Hence, if $r$ is a threshold between tracking and partial tipping or between partial and total tipping then
$$
W^u(\Gamma^s_-)\cap W^s(\Gamma^u_+)\neq \emptyset
$$
with non-transverse intersection along a unique trajectory, more precisely this means that at a typical point $p\in W^u(\Gamma^s_-)\cap W^s(\Gamma^u_+)$ we have
\begin{equation}
\dim\left(T_p W^u(\Gamma^s_-)\cap T_pW^s(\Gamma^u_+)\right)=2.
\label{eq:nontransverse}
\end{equation} 

\item If $r$ such that 
$$
W^u(\Gamma^s_-)\cap W^s(Z_+)\neq \emptyset
$$
then this is generically an isolated point in $r$, and hence a invisible tipping.
\end{itemize}

Figure~\ref{fig:diff_rates} illustrates some examples of numerical approximations showing trajectories and the relation between the stable manifold of the unstable equilibrium and unstable periodic orbit and the pullback attractors.

\begin{figure*}
	{\includegraphics[scale = 0.55]{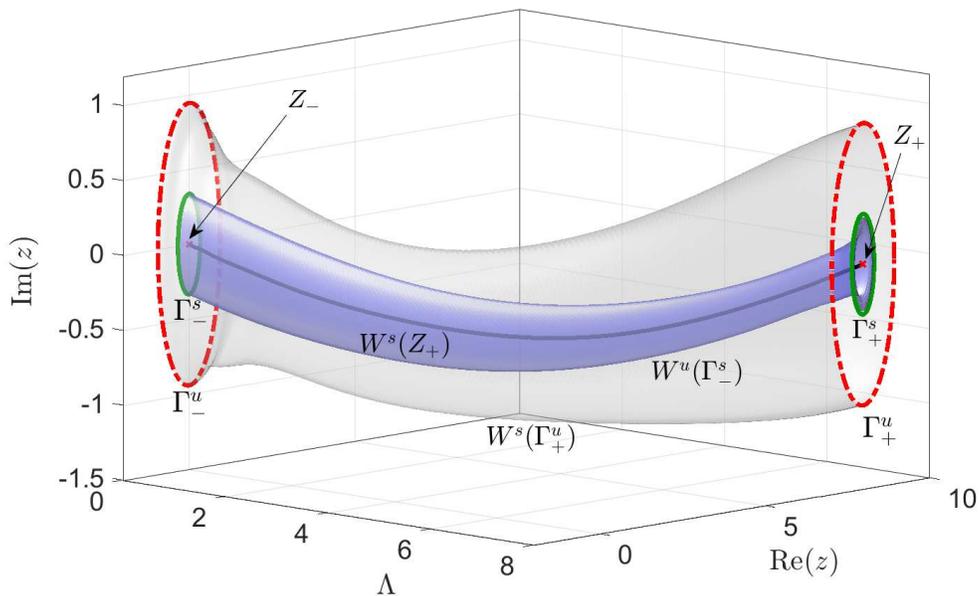}}
	
	\caption{Numerical approximations of the pullback attractor $\mathcal{A}^{[\Lambda,r,\Gamma_-^s]}$ ( for system (\ref{eq:zLambda}) ) for $a=0.1$, $r=0.1$, $b=1$, $\lambda_{\max}=8$ and $\omega=3$. The graph of the pullback attractor (inner dark tube) over $\Lambda$ is $W^u(\Gamma^s_-)$. In this case $r$ is chosen small enough that there is tracking of the periodic attractor according to Theorem~\ref{thm:tracking}. The outer grey tube shows $W^s(\Gamma^u_+)$ whilst the inner black line is $W^s(X_+)$. The red circles indicate $\Gamma^u_{\pm}$, the green circles indicate $\Gamma^s_{\pm}$ and the red points indicate $Z_{\pm}$.
	}     \label{fig:pb_example}
\end{figure*}

\begin{figure*}[!]
    	\subcaptionbox{} [0.48\linewidth]
    {\includegraphics[scale = 0.33]{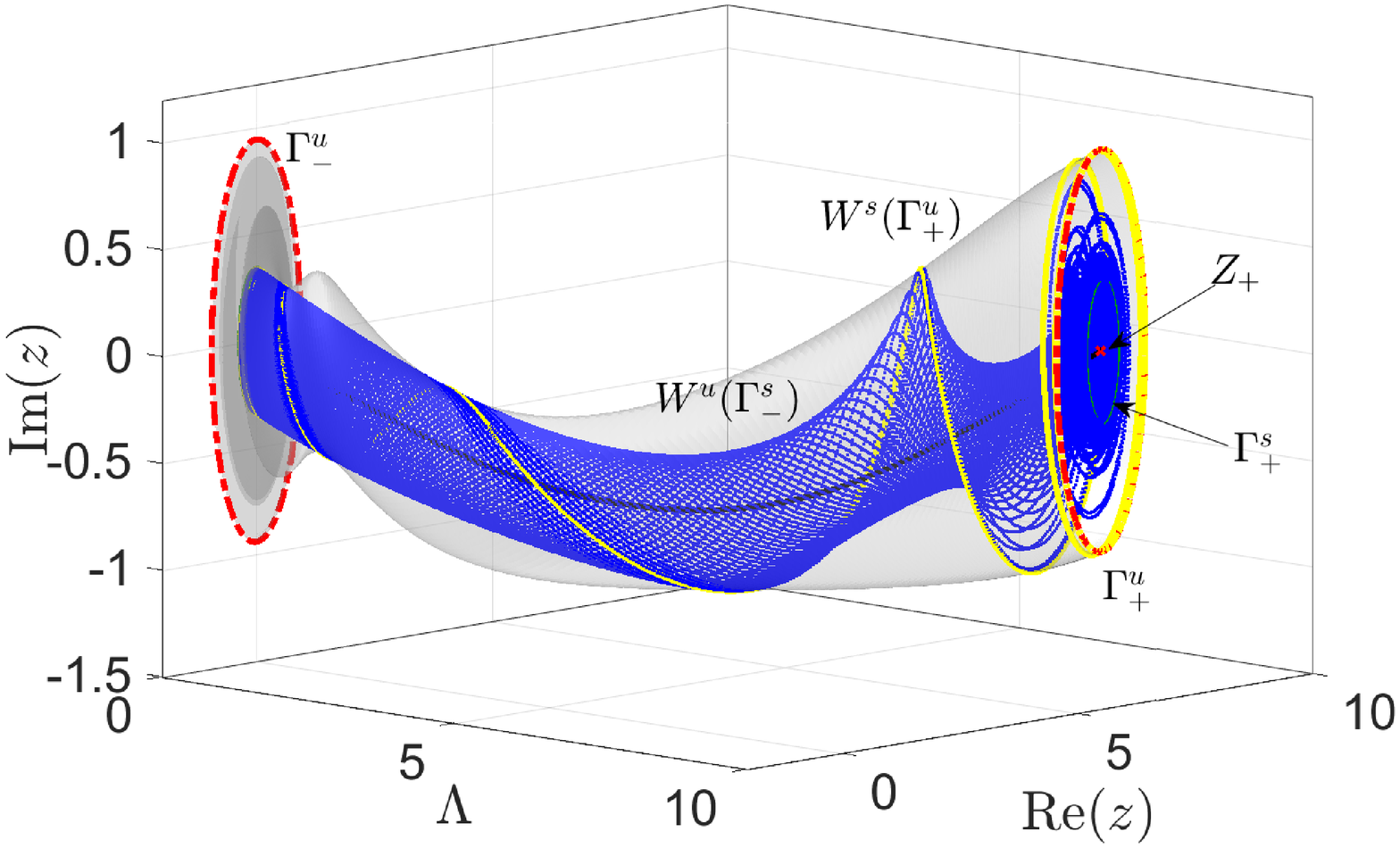}}
    	\subcaptionbox{} [0.48\linewidth]
    {\includegraphics[scale = 0.33]{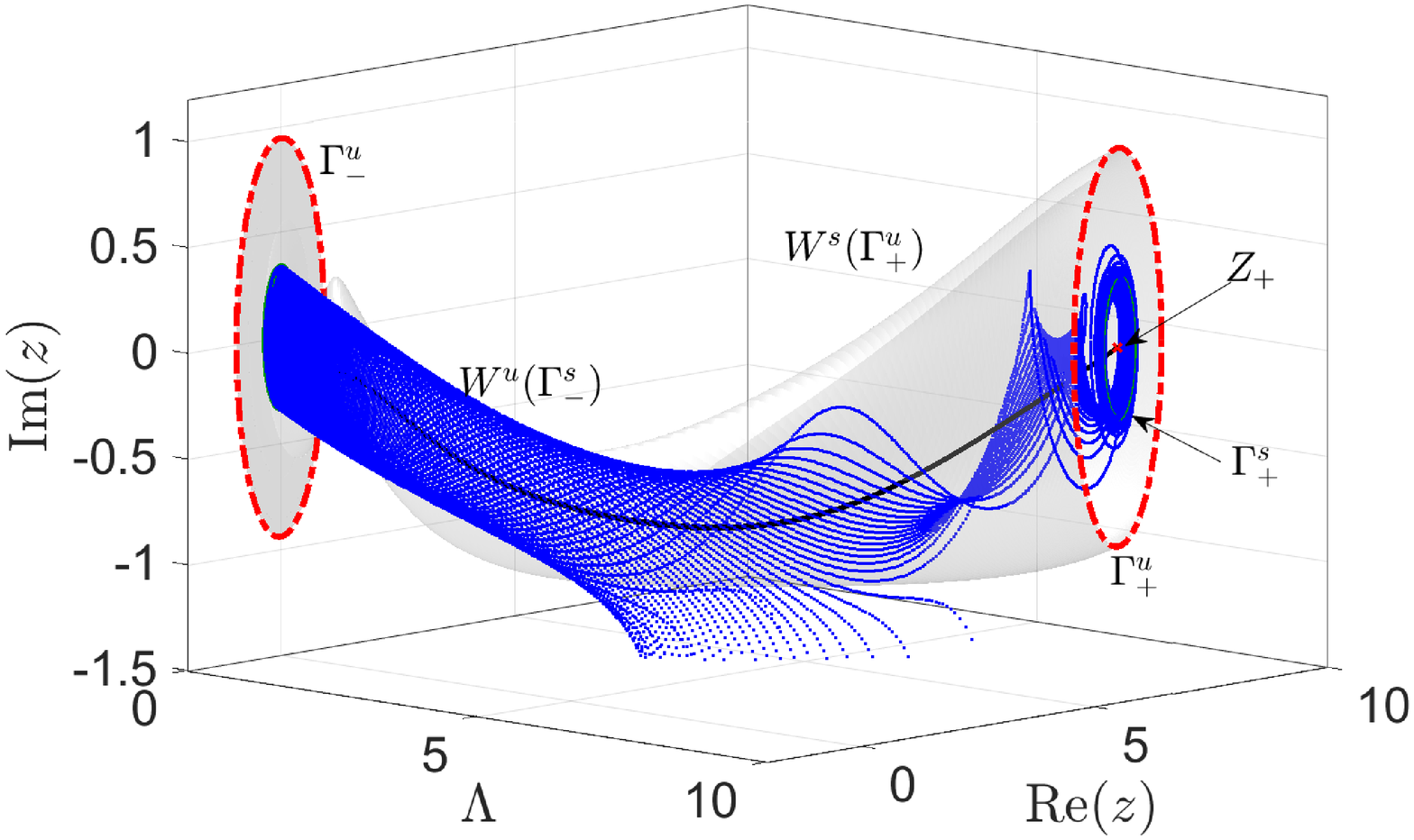}}
    	\subcaptionbox{} [0.48\linewidth]
    {\includegraphics[scale = 0.33]{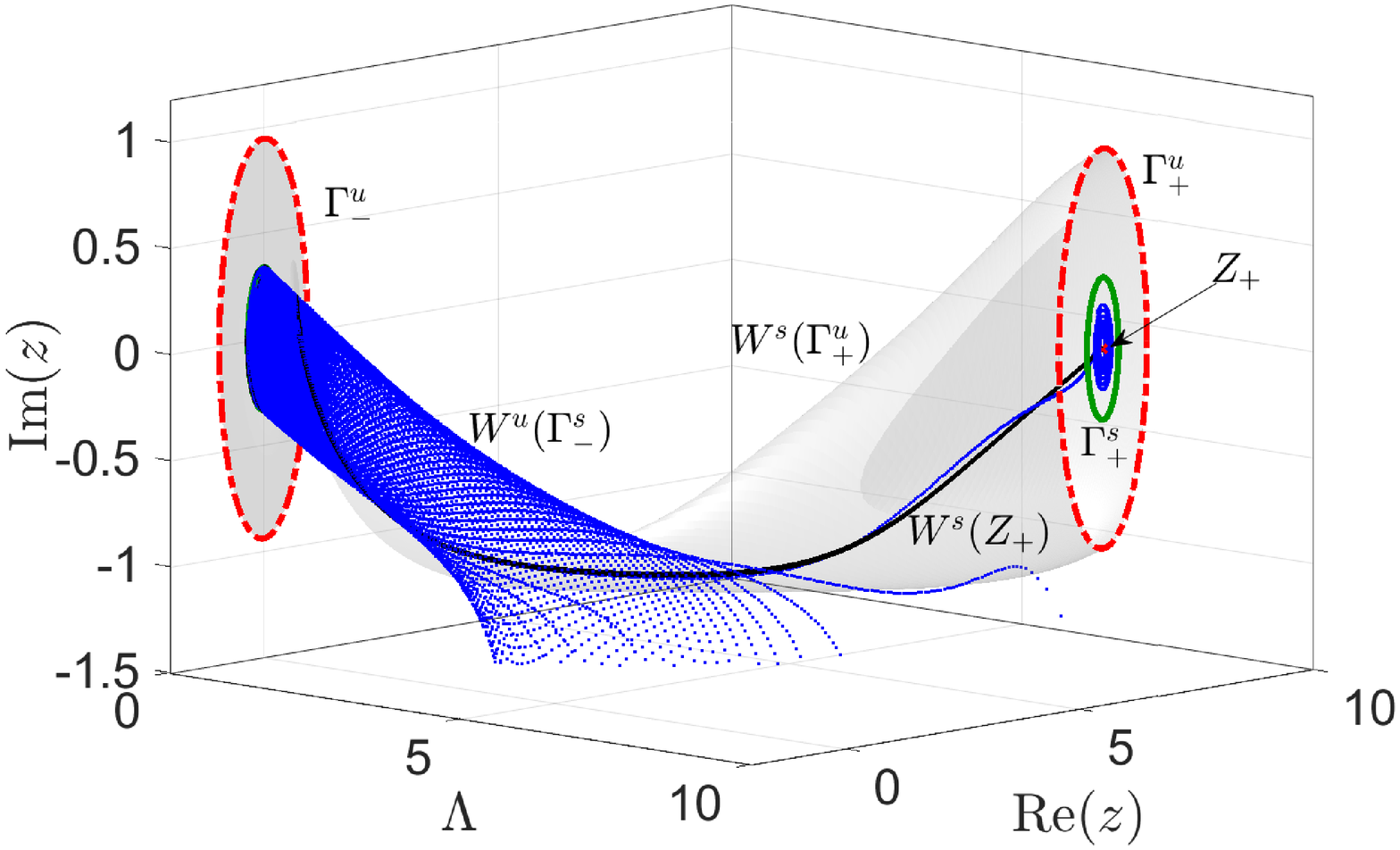}}
    	\subcaptionbox{} [0.48\linewidth]
    {\includegraphics[scale = 0.33]{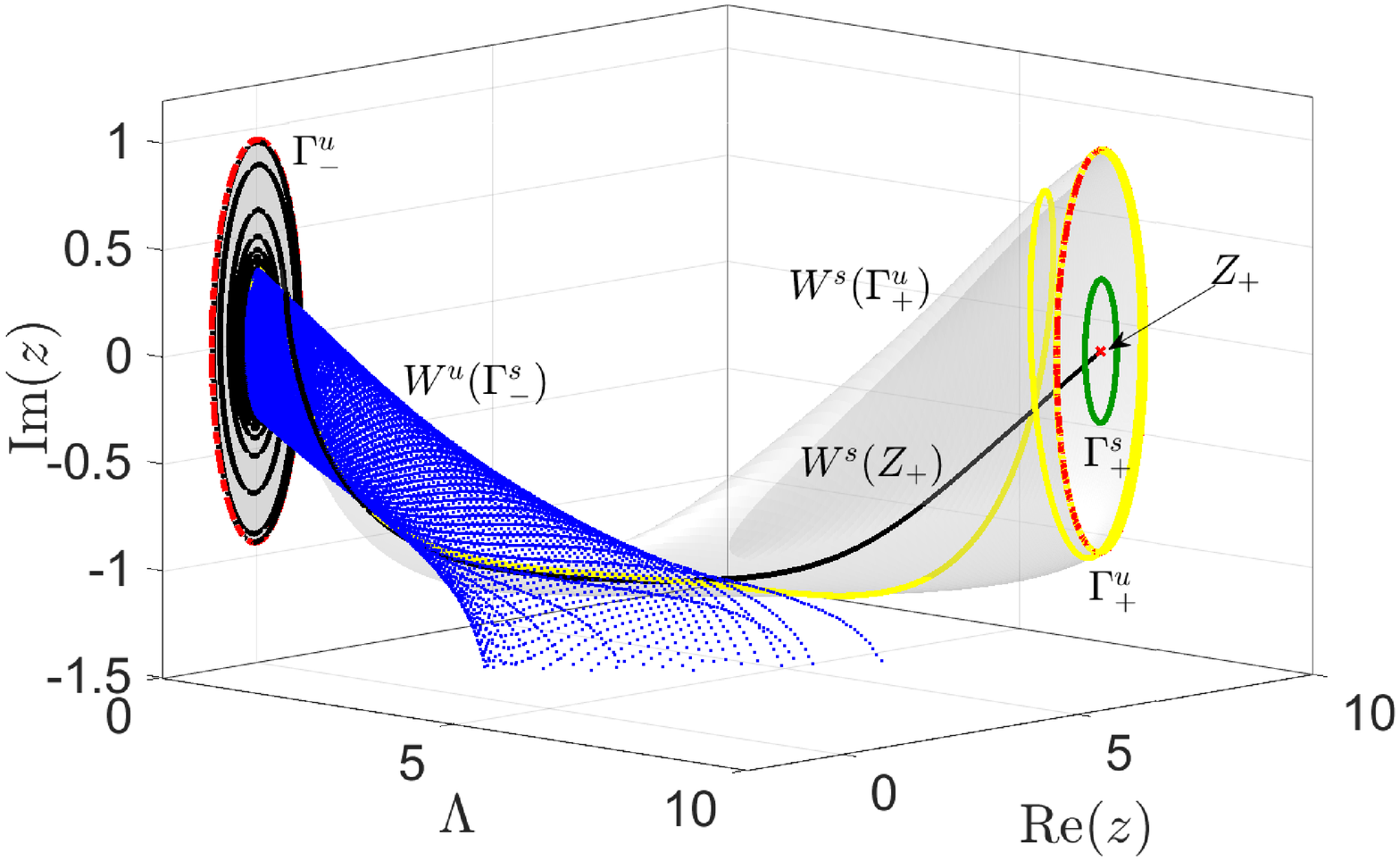}}

 \caption{Numerical approximations of the pullback attractor as in Figure~\ref{fig:pb_example} but for different examples  of tipping. (a) $r=0.13321$ at the threshold of partial tipping: there is a single connection (yellow) in $W^u(\Gamma^s_-)\cap W^s(\Gamma^u_+)$ with non-transverse intersection. (b) $r = 0.15$ in the region of partial tipping: some of the trajectories in blue on the pullback attractor track while others escape. (c) $r = 0.198422$, showing existence of a PtoE connection (black). (d) $r = 0.201226$ showing total tipping.
 }     \label{fig:diff_rates}
\end{figure*}

\subsection{Rate-induced tipping as bifurcations of PtoP and PtoE connections}

As outlined above, it is possible to find thresholds of rate-induced tipping by considering certain PtoP and PtoE heteroclinic connections, analogous to \citet[Proposition 4.1]{Perryman2015}. An efficient way of doing this is Lin's method \cite{Homburg2010} that involves solving three point boundary value problems with suitable boundary conditions that give the desired connection: see for example \cite{Knobloch2010a, Krauskopf2008, Zhang2012} for details. We outline our numerical implementation of Lin's method more details are included in Appendix~\ref{AppLinMeth}. Throughout we fix $b=1$, $\omega =3$ and $\lambda_{\max} =8$.

\citet{Zhang2012} give a systematic method to find a PtoP connection where the intersection between the tangent space of the unstable and the stable manifold is one dimensional. However, for our critical rates even though the PtoP connection is one-dimensional, the intersection of the tangent spaces is of dimension two, and solving \citet[equations (6) - (11)]{Zhang2012} give criteria for codimension-zero connections.  To find the critical rates of transition to partial and to total tipping we solve the adjoint variational equation (AVE) along the connection to allow us to test (\ref{eq:nontransverse}).

Let us denote the system \eqref{eq:zLambda} by 
\begin{equation} \label{bvp.ode}
\dot{w}=G(w;\mu)
\end{equation}
where $w(t)=(x(t),y(t),\Lambda(t)) \in \R^3$, $z(t)=x(t)+iy(t)$, $\mu=(a,r)\in \R^2$ and $G:\R^5 \rightarrow \R^3$ is the vector field of the system. The adjoint variational equation of a solution $w(t)$ of (\ref{bvp.ode}) at the parameter value $\mu_0$ is given by \citep{Homburg2010}: 
\begin{equation} \label{adjvareq}
\dot{u}= -G_u(w,\mu_0)^{\mathrm{tr}} u 
\end{equation} 
with solution $u(t)$, where $G_u(w,\mu_0)$ is the Jacobian matrix of the function $G(.,\mu_0)$ over $w$ and $A^\mathrm{tr}$ is the transpose of the matrix $A$. Let us assume that $T>0$ is a (sufficiently large) integration time, $g_1(\vartheta)\in \Gamma_-^s$, $g_2(\varphi)\in\Gamma_+^u$, $\gamma^{\pm}_{s,c,u}$ give the stable/center/unstable eigendirections of  $\Gamma_-^s$, $\Gamma_-^u$ respectively for $0<\vartheta$, $\varphi<2\pi$, and ${v_{1,2u}}$ are the unstable eigenvectors of $Z_+$. We can write the BVPs of the relevant connections as the following:

We locate and continue a PtoE connection $W^u(\Gamma^s_-)\cap W^s(X_+)\neq \emptyset$ (corresponding to invisible tipping) by choosing a section $\Lambda=\lambda_{\max}/2$ and a Lin basis vector $\ell$ and solving
\begin{equation}
\begin{aligned}
 \dot{w}^-(s) &= T G(w^-(s);\mu), \\
 \dot{w}^+(s) &= T G(w^+(s);\mu),
 \end{aligned}
\label{eq:PtoEEQ}
\end{equation}
 on $0<s<1$ with $T>0$ sufficiently large and boundary conditions 
\begin{equation}
\begin{aligned}
 0 &= \big\langle w^-(0)-g_1(\vartheta), \gamma_s^-    (\vartheta) \big\rangle , &
 0 &= \big \langle w^-(0)-g_1(\vartheta), \gamma_c^- (\vartheta) \big\rangle, \\
 0 &= \big\langle w^+(1)-Z_+, {v_{1u}} \big\rangle , &
 0 &= \big\langle w^+(1)-Z_+, {v_{2u}} \big\rangle, \\
 0 &= \big\langle w^-(1)-(0,0,\lambda_{\max}/2),(0,0,1)  \big\rangle , &
 \xi \ell & =w^+(0) - w^-(1).
\label{eq:PtoEBC}
\end{aligned}
\end{equation}

We locate a codimension zero PtoP connection in $W^u(\Gamma^s_-)\cap W^s(\Gamma^u_+)$  by similarly choosing a section $\Lambda=\lambda_{\max}/2$ and solving
\begin{align}
\begin{aligned}
\dot{w}^-(s) &= T G(w^-(s);\mu)\\
\dot{w}^+(s) &= T G(w^+(s);\mu),\label{eq:PtoPEQ}
\end{aligned}
\end{align}
on $0<s<1$ for some sufficiently large $T>0$ with boundary conditions
\begin{equation}
\begin{aligned}
0 &=\big\langle w^-(0)-g_1(\vartheta), \gamma_s^-(\vartheta) \big\rangle, &
0 &=\big \langle w^-(0)-g_1(\vartheta), \gamma_c^-(\vartheta) \big\rangle, \\
0&=\big\langle w^-(1)-g_2(\varphi), \gamma_u^+(\varphi) \big\rangle , &
0 &=\big\langle w^-(1)-g_2(\varphi), \gamma_c^+(\varphi) \big\rangle, \\
0 &=\big\langle w^-(1)-(0,0,\lambda_{\max}/2),(0,0,1) \big\rangle ,&
\xi \ell &= w^+(0) - w^-(1).
\label{eq:PtoPBVP1}
\end{aligned}
\end{equation}
This can be extended to find the codimension one PtoP connection (corresponding to a boundary between partial tipping and either tracking or total tipping) by solving (\ref{eq:PtoPEQ},\ref{eq:PtoPBVP1}) and in addition the adjoint variational equation
\begin{align}
\begin{aligned}
\dot{u}^-(s) &= -TG_u(w^-(s),\mu)^\mathrm{tr} u^-(s)\\
\dot{u}^+(s)&= -TG_u(w^+(s),\mu)^\mathrm{tr} u^+(s)\label{eq:PtoPEQAD}
\end{aligned}
\end{align}
with boundary conditions
\begin{equation}
\begin{aligned}
0 &=\big\langle u^-(0),\gamma^-_u(\vartheta) \big\rangle,&
0 &=\big\langle u^-(0),\gamma^-_c(\vartheta) \big\rangle,\\
0 &=\big\langle u^+(1),\gamma^+_s(\varphi) \big\rangle,&
0 &=\big\langle u^+(1),\gamma^+_c(\varphi) \big\rangle,\\
0 &=u^-(1)-u^+(0),&
1 &=\big\langle u^-(1) , (1,0,0)\big\rangle.
\label{eq:PtoPBVP2}
\end{aligned}
\end{equation}
More details are in Appendix~\ref{AppLinMeth}: note that $\xi$ is a parameter that is determined by solving the BVP: one can think of $\xi(r,a): \mathbb{R}^2 \rightarrow \mathbb{R}$ as a function whose zeros give the desired connections.   

\begin{figure*}[!]

   	\subcaptionbox{\label{fig:paraplane1}} [0.32\linewidth]
    {\includegraphics[scale = 0.26]{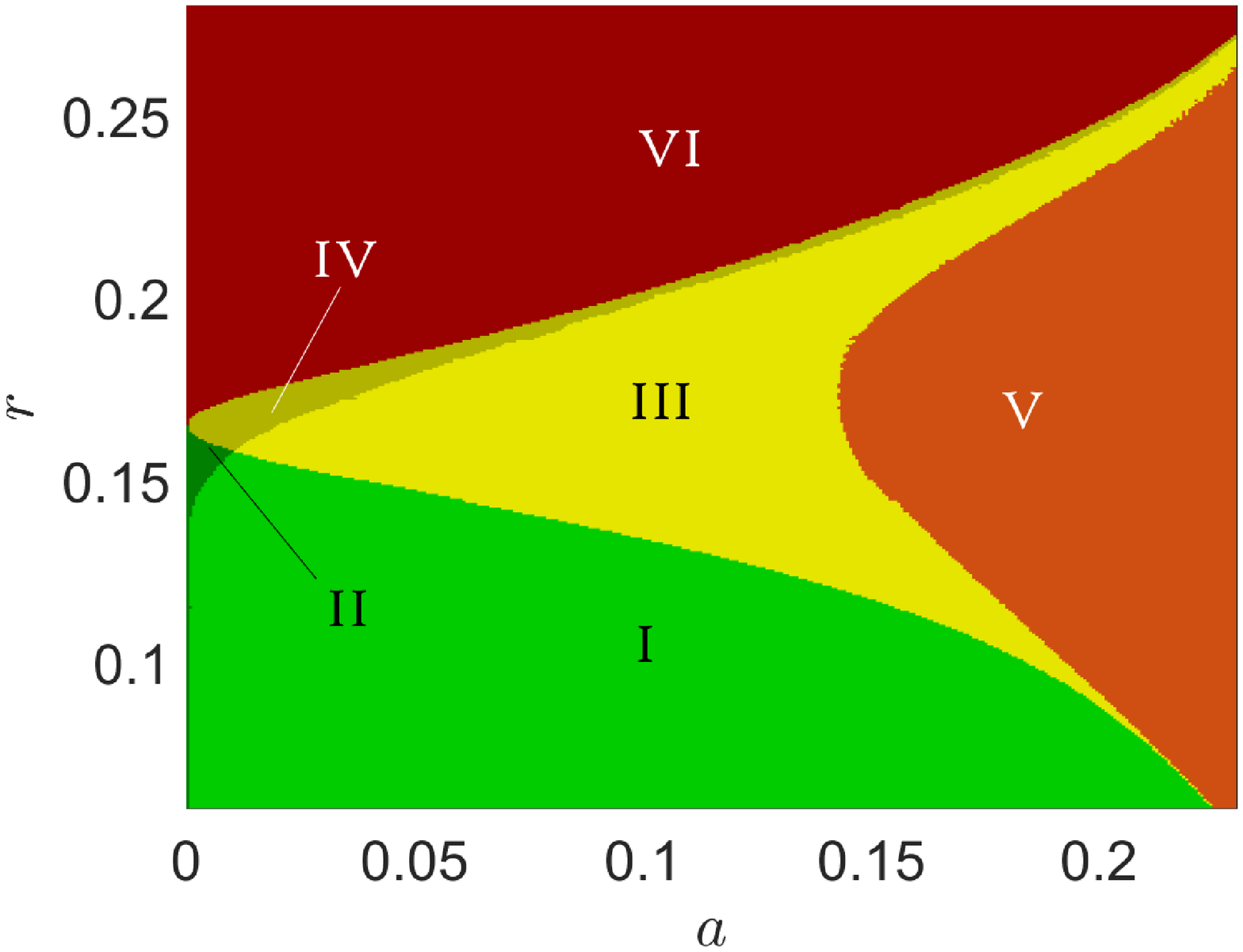}}
    	\subcaptionbox{\label{fig:paraplane2} } [0.32\linewidth]
    {\includegraphics[scale = 0.26]{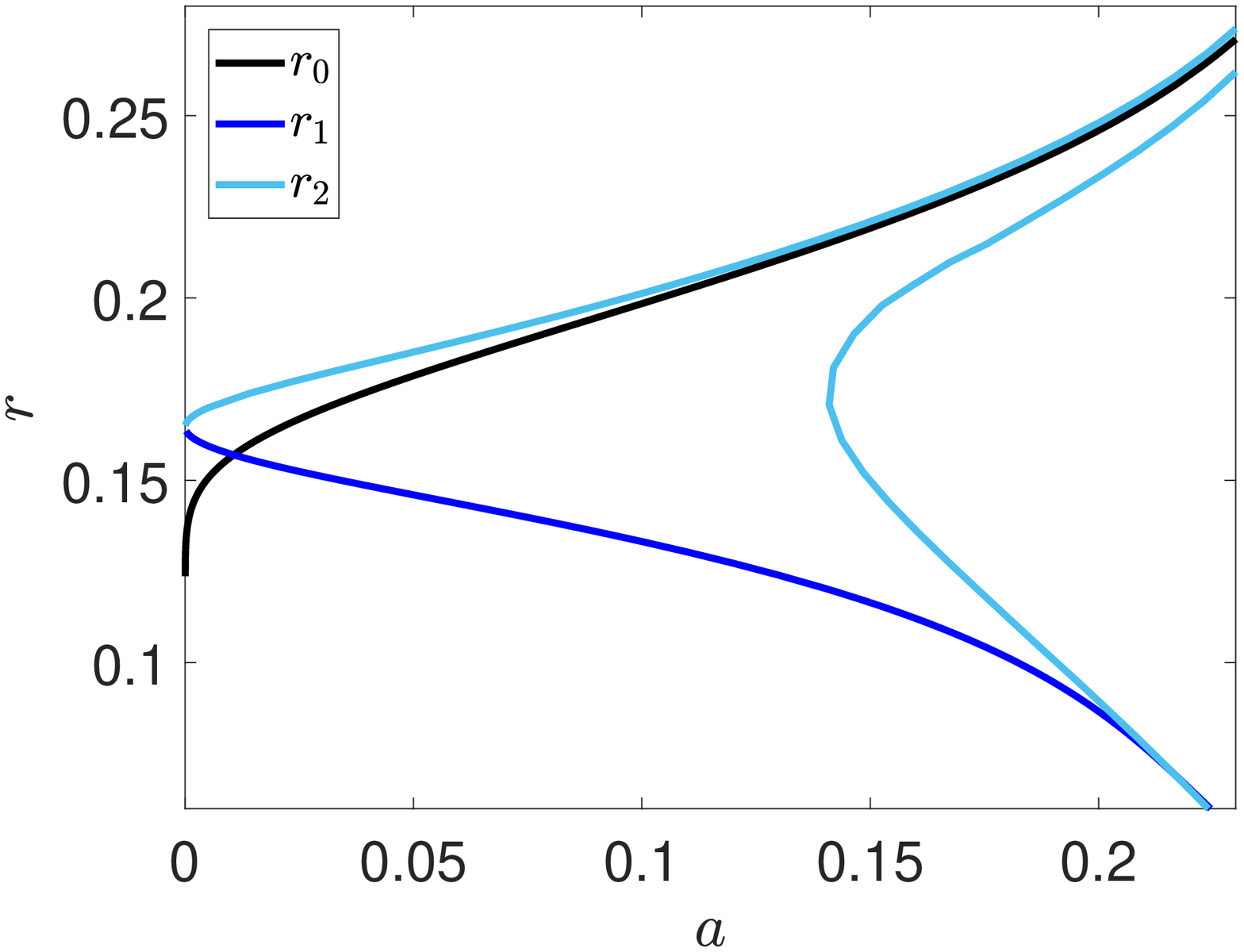}}
    \subcaptionbox{\label{fig:paraplane3}} [0.32\linewidth]
    {\includegraphics[scale = 0.26]{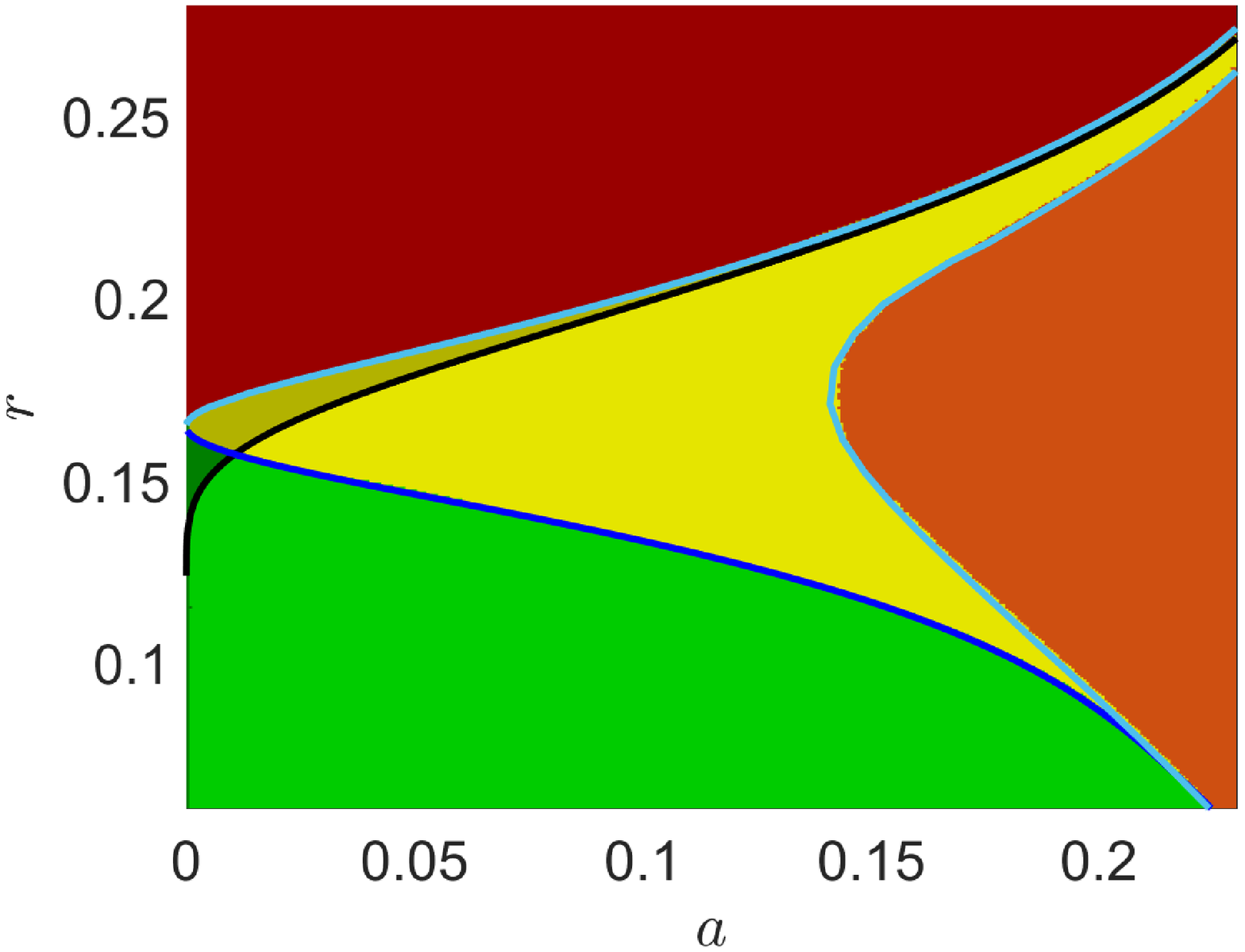}}
    \caption{The two parameter plane of system \eqref{eq:zLambda} showing regions of different tracking/tipping behaviour  (a) is calculated by directly approximating a collection of initial conditions on the pullback attractor and determining their fate under the dynamics of the system and shows six regions where the system has qualitatively different behaviour (see Figure \ref{fig:manifsect}). The curves in (b) are calculated using Lin's method and show the locations of these transitions: $r_{1,2}$ are the thresholds of partial and total tipping respectively and $r_0$ gives PtoE connection causing a invisible tipping for $0<a<0.0157$. In (c) they are superimposed. }
    
    \label{fig:paraplane}
\end{figure*}

Solving the system (\ref{eq:PtoPEQ},\ref{eq:PtoPBVP1},\ref{eq:PtoPEQAD},\ref{eq:PtoPBVP2}) allows one to determine and continue the codimension-one PtoP connections that give the thresholds of partial and total tipping. As initial solution we solve the codimension-zero problem (\ref{eq:PtoPEQ},\ref{eq:PtoPBVP1}) and continuing it along $r$ to arrive at a fold where the codimension-one connection exists. Figure~\ref{fig:paraplane} illustrates ($a,r$)-parameter plan for (\ref{eq:zLambda}) in the case $b=1$, $\omega=3$ and $\lambda_{\max}=8$ calculated by Lin's method and compares it with a direct shooting algorithm described in Appendix~\ref{appShooting}. Figure \ref{fig:manifsect} shows the behaviour of \eqref{eq:zLambda} in each different region of the parameter plan by looking at a section of the manifolds $W^u(\Gamma_-^s)$, $W^s(\Gamma_+^u)$ and $W^s(Z_+)$.

\section{Discussion}
\label{sec:discuss}

In this paper we discuss the phenomena of R-tipping from periodic orbits in the setting of parameter shift systems. We extend results of \citet[Theorems 2.2 and 2.4]{Ashwin2017} for equilibrium branches of attractors and show that  there exists a pullback attractor of \eqref{eq:nonautonomous} whose upper backward limit is contained within an attractor of the past limit system. Under additional assumptions on the stability of the branch we show that the pullback attractor tracks the branch for small rate $r>0$. Theorem~\ref{thm:tracking} states that, for a range of small values of $r$, the forward limit of the pullback attractor $A_{\infty}^{[\Lambda,r,A_-]}$ is the same. However, there is no guarantee of this with large enough $r$. Indeed, if there is rate-induced tipping then this is not the case. 

More generally, we note that the local pullback attractor can be used to classify a number of different types of tipping (see Definition~\ref{Def:tipping}) and use the example in Section~\ref{sec:example} to illustrate some differences. We have been able to present partial tipping, total tipping in addition to the tracking case. In order to investigate and continue the thresholds of partial and total tipping numerically for \eqref{eq:zLambda} we calculate PtoP and PtoE connections using Lin's method. 

The integration time $T$ in (\ref{eq:PtoEEQ}, \ref{eq:PtoPBVP1}) would need to be chosen to be proportional to $1/a$ near the Hopf ($a=0$) and $1/r$ near the fold of limit cycles ($a=0.25$) to resolve the details. Hence, any fixed $T$ will give errors in PtoE and PtoP connections in regions close to $a=0$ and $0.25$. Moreover, as $a\rightarrow 0.25$, $\|R_s - R_u\|\rightarrow 0$ which means it became very difficult for the pullback attractor $\mathcal{A}^{[\Lambda,r,\Gamma^s_-]}$ to track the branch $\Gamma^s(\Lambda(rt))$ even for very small $r>0$ (i.e as $a \rightarrow 0.25$, $r_1\rightarrow 0$ as well as $\|r_1-r_2 \|\rightarrow 0$).

The $(a,r)$-parameter plane (Figure~ \ref{fig:paraplane}) shows that the upper parts of regions III and IV of partial tipping thin out for $a>0.15$. We explain this as follows: the threshold of partial and total tipping get close together because of the fold of limit cycles and the PtoE connection curve is trapped between these. Even for relatively large rate the connection between $\Gamma_-^s$ and $Z_+$ is associated with partial tracking (partial tipping).
 
For practical reasons, it would be very useful to find warnings of tipping points, and ``early warning indicators'' have been developed in several cases (see for example \cite{Ditlevsen2010,Ritchie2016,Clements2016}). We mention in particular the work of \citet{Ritchie2016} which shows that even for R-tipping some of the most widely used early warning signals, like increase of the autocorrelation and variance, may be useful. Extending those results and applying them on partial tipping of attractors that are not simply equilibria is not straightforward. For example, although the phenomenon of partial tipping is quite clear if the attractors are  considered set-wise, from individual trajectories it is not possible to determine whether there is partial or total tipping. This is a challenging issue one has to tackle in order to develop early-warning signals for partial tipping.
 
Finally, we note that dealing with non-minimal attractors (i.e attractors that have proper sub attractors) could lead to a weak type of tracking. {\em Weak tracking} happen when the forward limit of the pullback attractor $A_{+\infty}^{[\Lambda,r,A_-]}$ included as a proper subset of the attractor of the future limit system $A_+$. There is no  possibility of weak tracking for a branch of periodic orbit attractors, simply because of minimality of the periodic orbit and Lemma \ref{lem:invariant limit}. If One consider more general branches of attractors however, this becomes a real possibility.

\begin{figure}[!]

   	\subcaptionbox{\label{fig:manifsect1}} [0.32\linewidth]
    {\includegraphics[scale = 0.27]{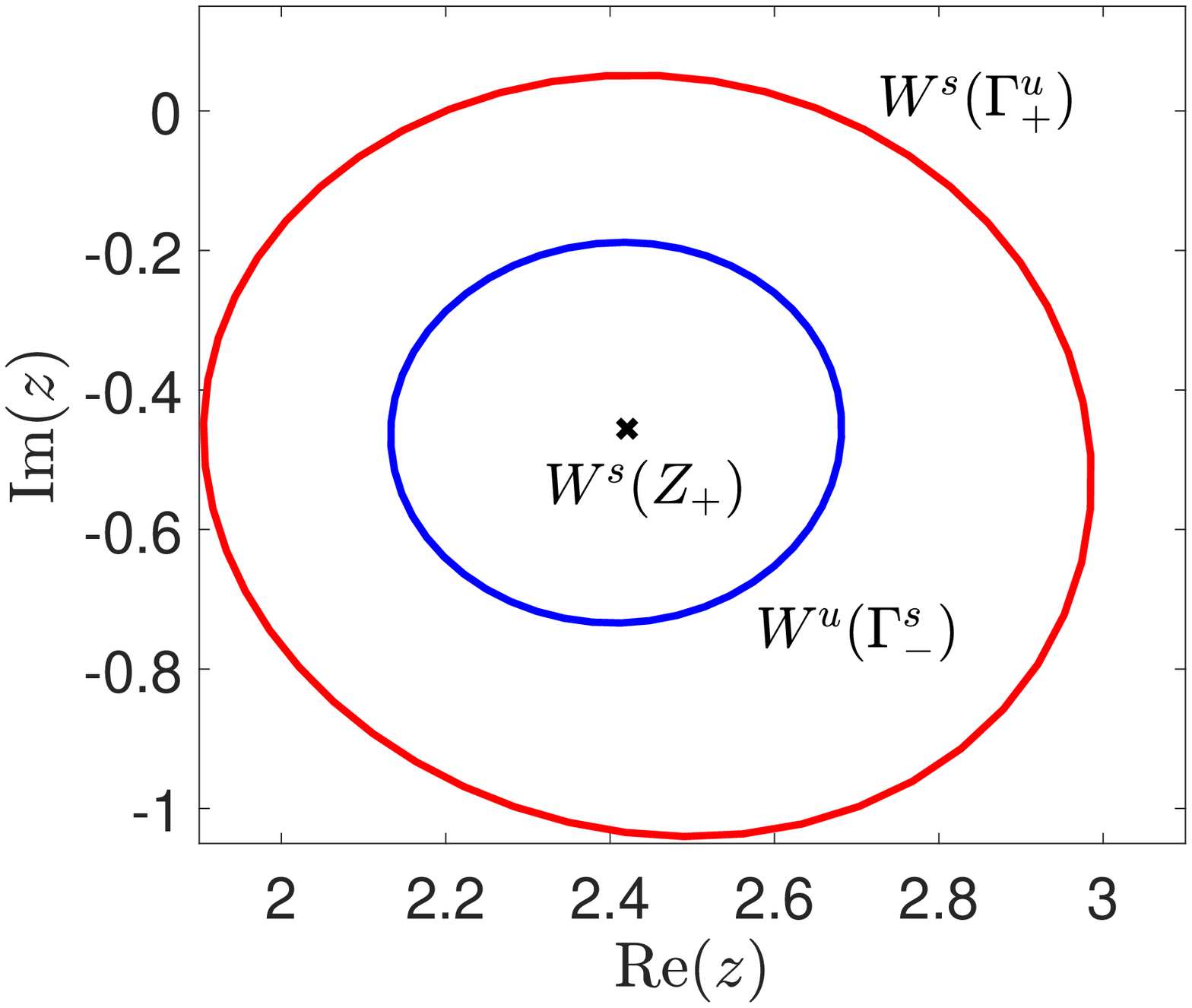}}
    	\subcaptionbox{\label{fig:manifsect2} } [0.32\linewidth]
    {\includegraphics[scale = 0.27]{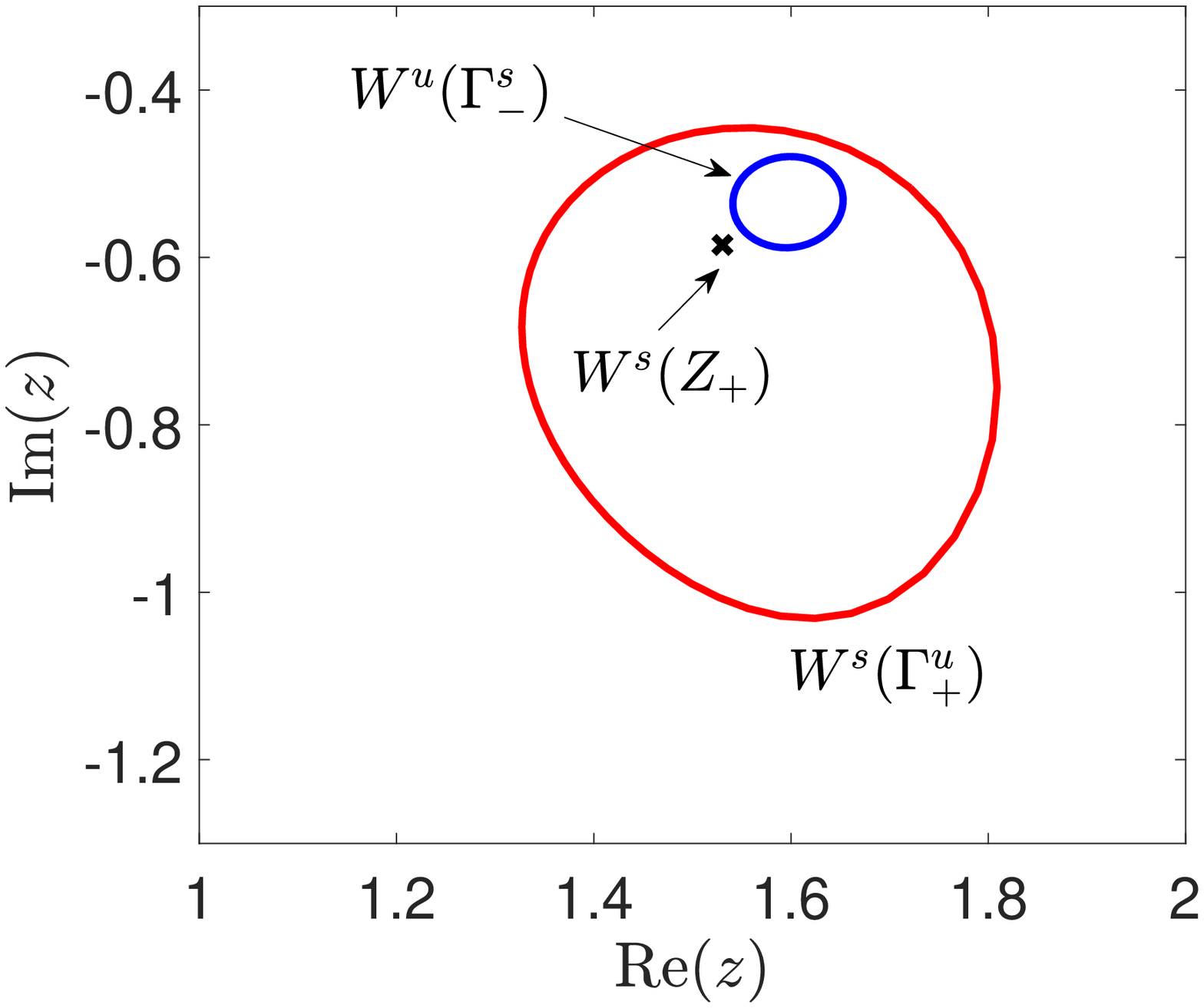}}    
    \subcaptionbox{\label{fig:manifsect3}} [0.32\linewidth]
    {\includegraphics[scale = 0.27]{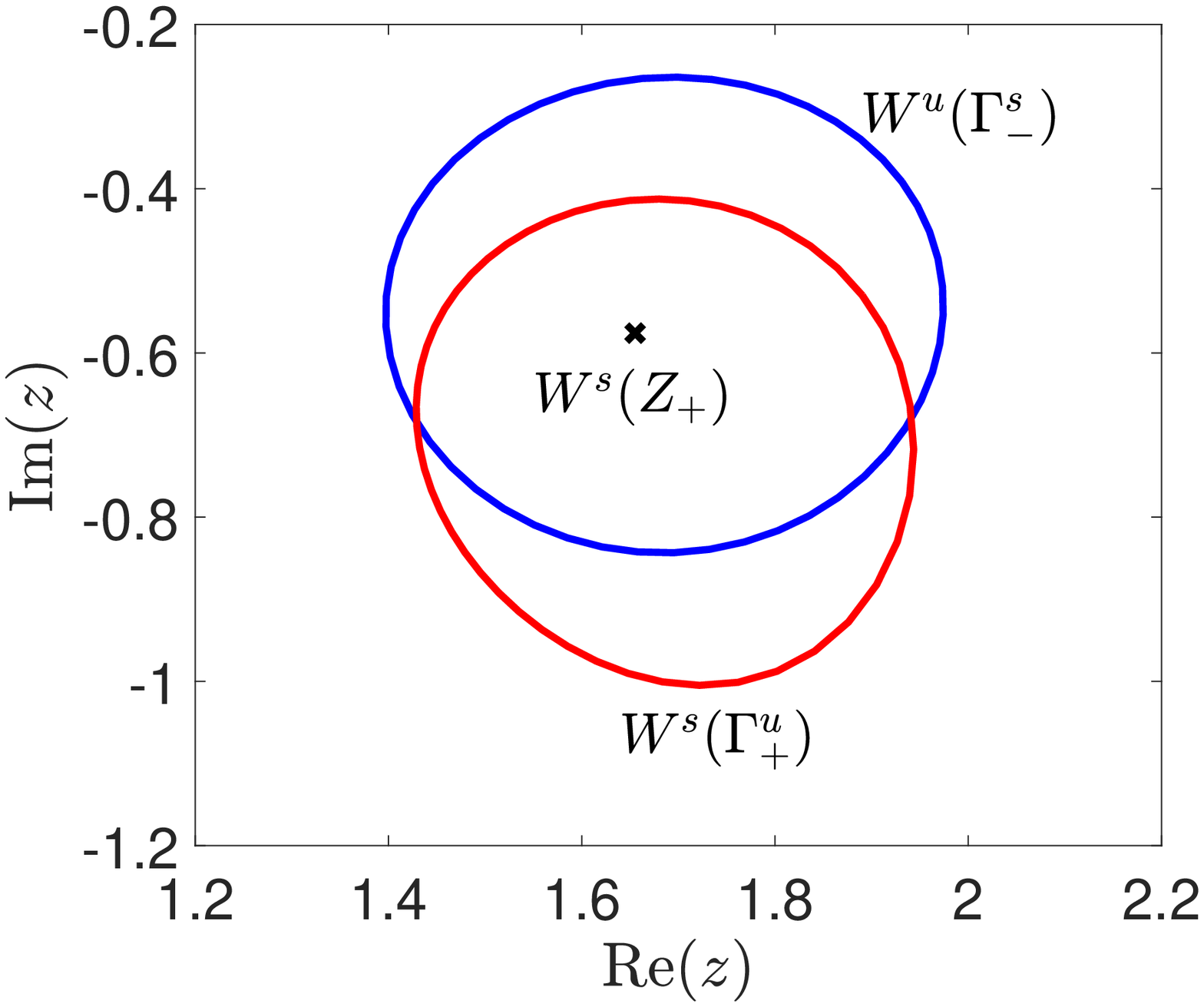}}

      \subcaptionbox{\label{fig:manifsect4}} [0.32\linewidth]
    {\includegraphics[scale = 0.27]{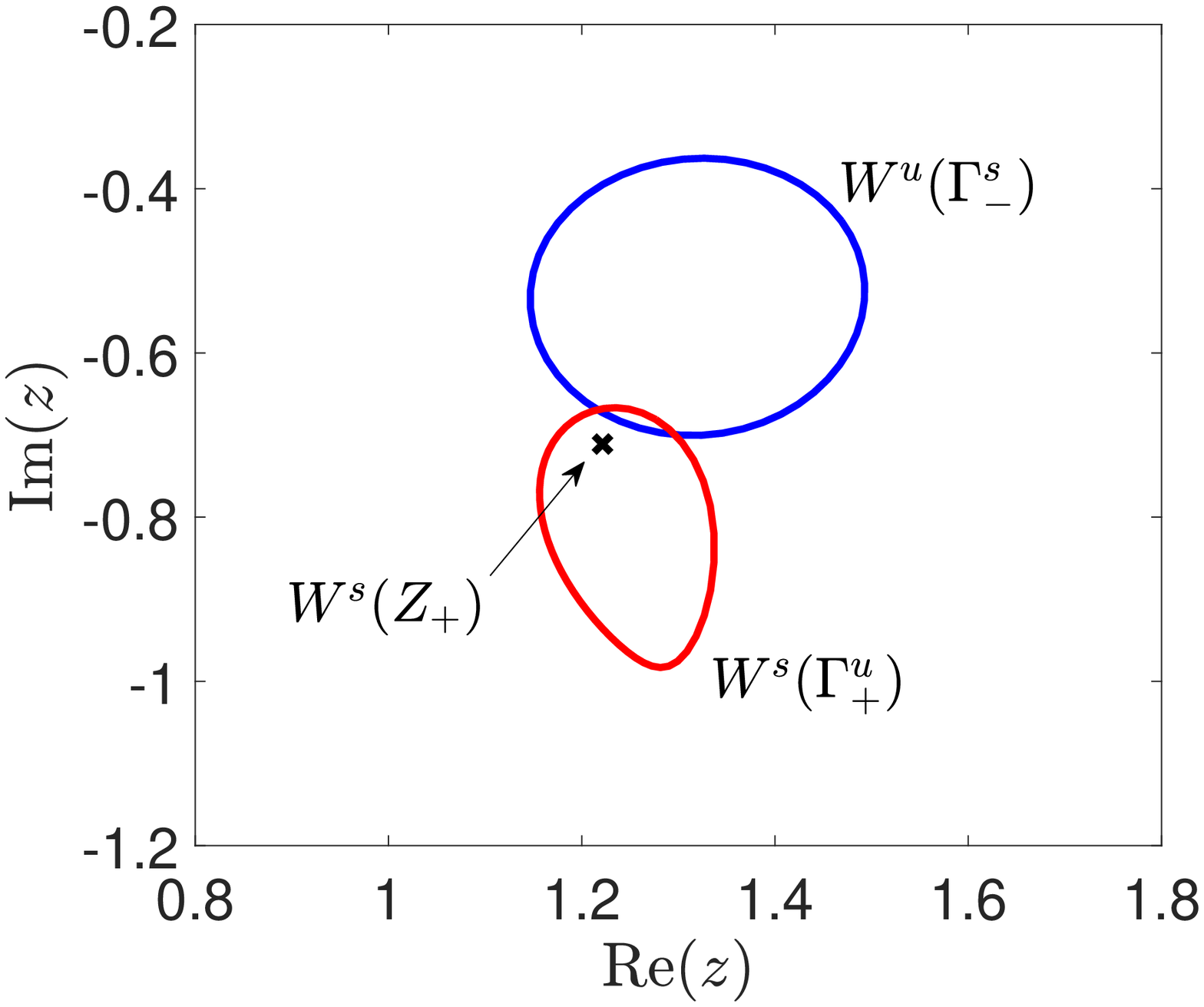}}    
    	\subcaptionbox{\label{fig:manifsect5} } [0.32\linewidth]
    {\includegraphics[scale = 0.27]{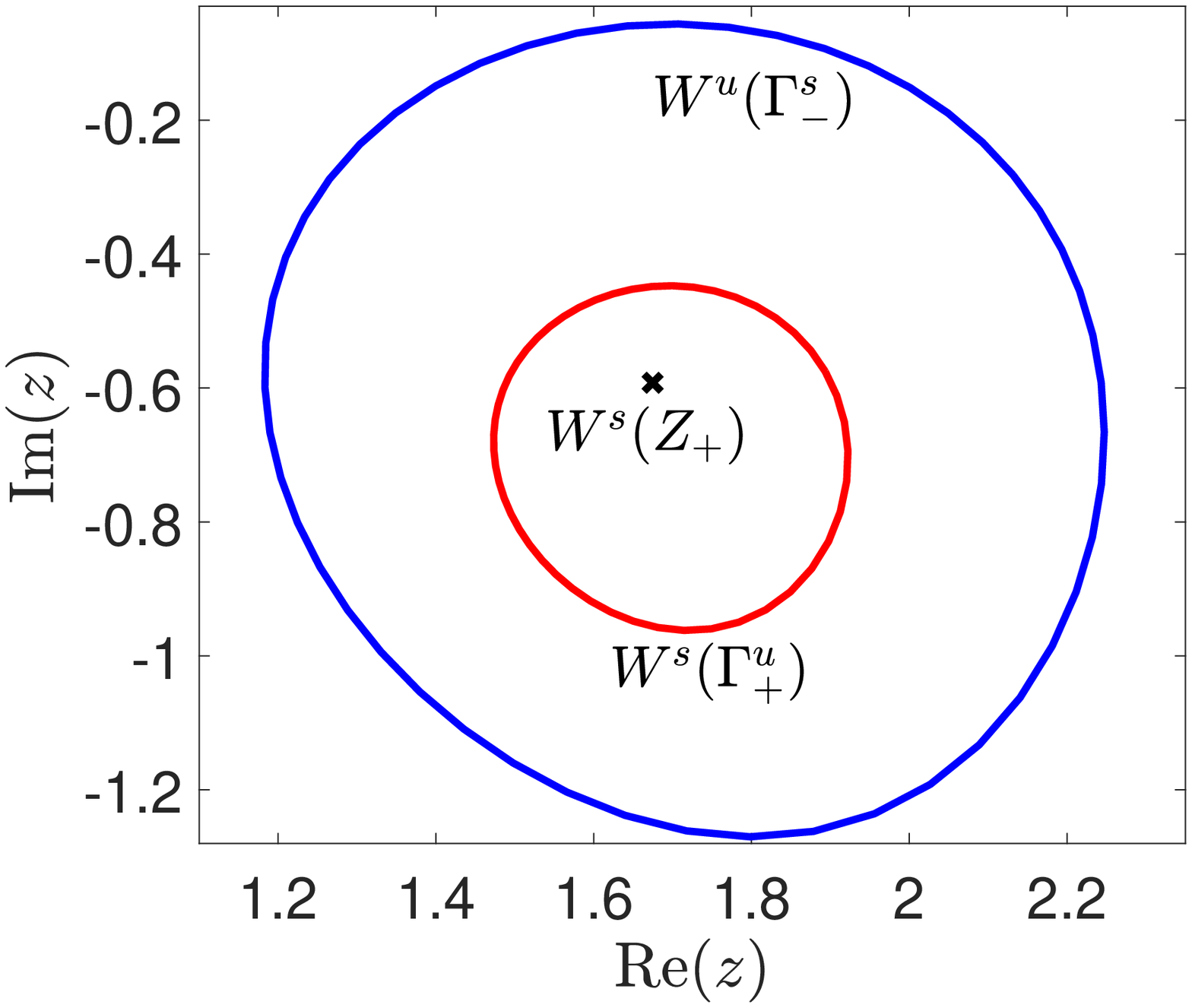}}
    \subcaptionbox{\label{fig:manifsect6}} [0.32\linewidth]
    {\includegraphics[scale = 0.27]{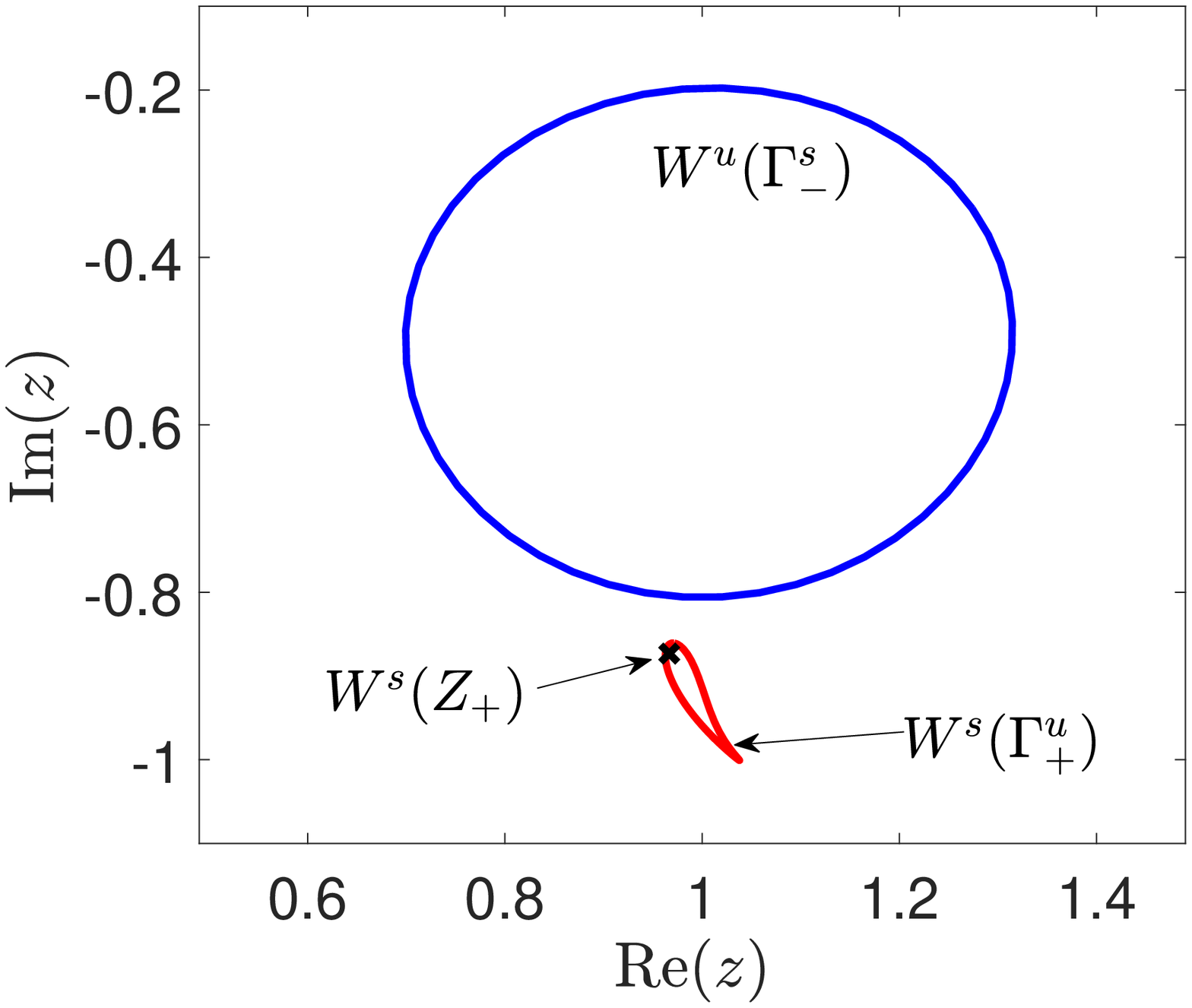}}
   \caption{A section (fixed $\Lambda\approx 1.257$) of the manifolds $W^u(\Gamma_-^s)$ (in blue). $W^s(\Gamma_+^u)$ (in red) and $W^s(Z_+)$ (in black) showing the topological behaviours of their stable intersections for regions I-VI shown in Figure \ref{fig:paraplane1}. The values of the parameters are (a) $a=0.1$, $r=0.1$ region I, (b) $a=0.005$,$r=0.157$ region II, (c) $a=0.1$, $r=0.15$ region III, (d) $a=0.04$, $r=0.18$ region IV, (e) $a=0.2$, $r=0.15$ region V and finally (f) $a=0.1$, $r=0.21$ region VI.}
    \label{fig:manifsect}
\end{figure}

\begin{acknowledgements}
HA's research is funded by the Higher Committee For Education Development in Iraq (HCED Iraq) grant agreement No D13436. PA's research is partially supported by the CRITICS Innovative Training Network, funded by the European Unions Horizon 2020 research and innovation programme under the Marie Sklodowska-Curie grant agreement No 643073. We would like to thank the following for their valuable comments on this research at various stages: Jan Sieber, Bernd Krauskopf, Ulrike Feudel, Mark Holland, Damian Smug, Courtney Quinn, Paul Ritchie, Sebastian Wieczorek and James Yorke.
\end{acknowledgements}

\appendix

\section{Proof of Theorem 2.2}
\label{AppProofTh2.2} 
\begin{proof} 
To show that $A_{-\infty}^{[\Lambda,r,A_-]}\subset A_-$, choose $\tilde{\eta}$ as in Lemma~\ref{lem:absorbing}, and pick any $\eta \in (0,\tilde{\eta}]$. By Lemma~\ref{lem:uniqueness}, the upper backward limit of $\mathcal{A}^{[\Lambda,r,A_-]}$ can be uniquely defined as: 
$$
A_{-\infty}^{[\Lambda,r,A_-]} = \bigcap_{\tau > 0} \overline{\bigcup_{t \leq -\tau} A_t^{[\Lambda,r,A_-]}}=  \bigcap_{\tau > 0} \overline{ \bigcup_{t \leq -\tau \atop s<t} \Phi\left(t,s,\mathcal{N}_\eta(A_-)\right)}.
$$
By Lemma~\ref{lem:absorbing}, for all $\delta>0$ there exists $\tau , \tilde{\tau}>0$ such that 
$$
\bigcup_{t \leq -\tau \atop s<t-\tilde{\tau}} \Phi(t,s,\mathcal{N}_\eta (A_-)) \subset \mathcal{N}_\delta(A_-)
$$ 
which gives
$$ 
A_{-\infty}^{[\Lambda,r,A_-]} = \bigcap_{\tau > 0} \overline{\bigcup_{t \leq -\tau \atop s<t} A_t^{[\Lambda,r,A_-]}} \subset \mathcal{N}_\delta(A_-)
.$$ 
Recall that holds for all $\delta > 0$, which in turn implies that
$A_{-\infty}^{[\Lambda,r,A_-]} \subset \overline{A_-} = A_-$.

To show that \eqref{eq:pbattractor} is a pullback attractor, we need to show it is compact, invariant and attracts a neighbourhood. For all $t \in \R$, ${A}_t^{[\Lambda,r,A_-]}$ is intersection of closed sets, which implies that it is closed. To show that it is compact, we just need to show it is bounded. By using Lemma~\ref{lem:absorbing} again, $\Phi(s_2,s_1,\mathcal{N}_\eta(A_-)) \subset \mathcal{N}_\eta(A_-)$ for all $s_1<s_2-\tilde{\tau}< -\tau$, by the cocycle property of $\Phi$ we get:
$$
\bigcup_{s<-\tau} \Phi(t,s,\mathcal{N}_\eta(A_-)) \subset \Phi(t,-\tau,\mathcal{N}_\eta (A_-)).
$$ 
Now since $\Phi(t,s,.)$ is a diffeomorphism for all $t,s \in \R$, $\Phi(t,-\tau,\mathcal{N}_\eta(A_-))$ is bounded, and so $\bigcup_{s<-\tau} \Phi(t,s,\mathcal{N}_\eta(A_-))$ is bounded. Hence, ${A}^{[\Lambda,r,A_-]}_t=\bigcap_{\tau>0}\overline{\bigcup_{s<-\tau} \Phi(t,s,\mathcal{N}_\eta(A_-))}$ is bounded. Therefore, ${A}^{[\Lambda,r,A_-]}_t$ is compact for all $t\in\R$. 

To prove $\mathcal{A}^{[\Lambda,r,A_-]}$ is invariant note that
\begin{eqnarray*}
\Phi(t,s,{A}^{[\Lambda,r,A_-]}_s)&=& \Phi \left(t,s,\bigcap_{\tau>0}\overline{\bigcup_{k<-\tau} \Phi(s,k,\mathcal{N}_\eta(A_-))}\right) \\
&=& \bigcap_{\tau>0} \overline{\bigcup_{k<-\tau} \Phi \left(t,s, \Phi(s,k,\mathcal{N}_\eta(A_-))\right)}  \\
&=&\bigcap_{\tau>0} \overline{\bigcup_{k<-\tau} \Phi \left(t,k,\mathcal{N}_\eta(A_-)\right)} \hspace{3.3cm} \\
&=& {A}^{[\Lambda,r,A_-]}_t
\end{eqnarray*}   
for all $t>s$ (we use the property that $\Phi(t,s,\cdot)$ is a diffeomorphism for all $t,s$).

To show that $\mathcal{A}^{[\Lambda,r,A_-]}$ attracts an open set $U$ in pullback sense, let  $U=\mathcal{N}_\eta(A_-)$, $t\in \mathbb{R}$ with $\eta$ as before, and define
$$
B_{\tau,t}:= \overline{\bigcup_{k<-\tau} \Phi (t,k,U)}.
$$
Note that $A_t^{[\Lambda,r,A_-]}=\bigcap_{\tau>0} B_{\tau,t}$ and $B_{s,t}\subset B_{\tau,t}$ for any $\tau<s$. Moreover, $d(B_{\tau,t},A_t^{[\Lambda,r,A_-]})\rightarrow 0$ as $\tau\rightarrow \infty$. Using Lemma~\ref{lem:absorbing} we have that
$$
\Phi(t,s,U)\subset B_{\tau,t}
$$
for all sufficiently negative $s$ (depending on $t$ and $\tau$). Hence for such $s$
$$
d(\Phi(t,s,U),B_{\tau,t})=0
$$
Hence
$$
\lim_{s\rightarrow -\infty} d(\Phi(t,s,U),A_t)=0
$$
and thus $\mathcal{A}^{[\Lambda,r,A_-]}$ is a pullback attractor. 
\end{proof}

\section{Approximating PtoP and PtoE connections using Lin's method}
\label{AppLinMeth}

We consider some Lin problems for our system where there are connections between the saddle objects $\Gamma_-^s$ and $\Gamma_+^u$ and $Z_+$. We are looking for connections between $\Gamma_-^s$ in the past and $\Gamma_+^u$, $Z_+$ on the future. The unstable and stable manifold, $W^u(\Gamma_-^s)$ and $W^s(\Gamma_+^u)$, $W^s(Z_+)$ are of dimensions 2, 2 and 1 respectively. Assuming there exist a connection $Q$ then for all point $q \in Q$ we have the following: 
$$
\dim\big(T_qW^u(\Gamma_-^s) \bigcap T_qW^s(\Gamma_+^u)\big)=2
$$
$$
\dim\big(T_qW^u(\Gamma_-^s) \bigcap T_qW^s(Z_+)\big)=1.
$$
We set the Lin section $\Sigma$, which is two dimensional liner space, half way between:
$$
\Sigma = \big\{w\in\R^3: \big\langle w-(0,0,\lambda_{\max}/2),(0,0,1) \big\rangle = 0 \big\}.
$$
The connection orbit $Q$ intersects $\Sigma$ transversely. i.e. $Q=Q^-\bigcup Q^+$ where: 
$$
Q^- = \{w^-(t):t\leq0 \}\subset W^u(\Gamma_-^s) \ \ \text{where} \ \  w^-(1)\in\Sigma,
$$
$$
Q^+ = \{w^+(t):t\geq0 \}\subset W^s(\Gamma_-^u) \ \ \text{where} \ \  w^+(0)\in\Sigma.
$$

Now we define the ``Lin gap" $\eta:=w^-(1)-w^+(0) \in \Sigma$. Lin's method require that $\eta$ lies in a fixed $d\leq\dim(\Sigma)-1$ dimensional liner space $L$, which satisfy the following condition,\cite{Krauskopf2008}
\begin{equation}
\label{Linspace} 
\dim(W^- \oplus W^+ \oplus L)=\dim(\Sigma)
\end{equation}
where $W^- = T_{w^-(0)}W^u(\Gamma_-^s)\bigcap T_{w^-(0)}\Sigma$ and $W^+ = T_{w^+(0)}W^s(\Gamma_+^u)\bigcap T_{w^+(0)}\Sigma$. 
The choice of $L$ could be done by considering the adjoint variational equation along the solution $Q$\cite{Zhang2012}, however the Lin space can be chosen arbitrarily as long as (\ref{Linspace}) is satisfied.  The definitions of $Q^-$ and $Q^+$ as well as condition (\ref{Linspace}) are formulated to investigate the PtoP connection between $\Gamma_-^s$ and $\Gamma^u_+$. However, it still applicable to the PtoE connection between $\Gamma_-^s$ and $Z_+$ with changing $\Gamma_+^u$ to $Z_+$ in each of them.

Note we also need approximations of the eigendirections for the periodic orbits: given a periodic solution $\Gamma=\left\{g(t): 0<t< T_\Gamma \right\}$ of the system \eqref{bvp.ode} with period $T_\Gamma$, the eigendirections $\gamma_{s,c,u}$ and Floquet multiplies $\beta_{s,c,u}$ are obtained as solutions of
\begin{eqnarray}
\begin{aligned}
\dot{\gamma}_{s,c,u} &= T_\Gamma \, G_u(g(s);\mu)\gamma_{s,c,u},  \\ 
\gamma_{s,c,u}(1)&= \beta_{s,c,u} \, \gamma_{s,c,u}(0),&
1 &= \big\langle \gamma_{s,c,u}(0), \gamma_{s,c,u}(0) \big\rangle.
\label{eq:eigdirc}
\end{aligned}
\end{eqnarray}
for $0<s<1$.

We implement this method as follows:
\begin{itemize}
\item Solving the boundary value problem \eqref{eq:eigdirc} numerically by using the {\tt bvp5c} MATLAB solver gives the eigendirections for $\Gamma_-^s$ and $\Gamma_+^u$ which can be used to formulate the projection conditions in (\ref{eq:PtoEBC}, \ref{eq:PtoPBVP1}, \ref{eq:PtoPBVP2}).
\item We formulate the solution of (\ref{eq:PtoEEQ}, \ref{eq:PtoPEQ}) as MATLAB functions that return $\xi(r,a)$, using the same boundary value solver. We use $(0,1,0)$ as a basis for the Lin space $L$.
\item We consider $\xi: \mathbb{R}^2 \rightarrow \mathbb{R}$ as smooth real valued function that by finding its zero one can find the desired connections. We did that by using Newton-Raphson iteration with tolerance $10^{-5}$ and defining the derivative of $\xi$ by finite difference with step size $10^{-4}$. 
\item Continuing the zero set of $\xi (r,a)$ in the $(a,r)$-plane by pseudo-arclength continuation gives the curves in Figures~\ref{fig:paraplane2} and \ref{fig:paraplane3}. 
\end{itemize}

\section{Finding the tracking/tipping regions by using shooting method}
\label{appShooting}

The tracking/tipping regions of \eqref{eq:zLambda} shown in Figure~\ref{fig:paraplane1} and \ref{fig:paraplane3} are found using a shooting method as follows: 
\begin{itemize}
\item We start with $M$ evenly spaced initial conditions near the periodic orbit $\Gamma_-^s$ and integrate \eqref{eq:zLambda} forward in time using the {\tt ode45} MATLAB solver. We vary $M$ depending on the value of $r$. As $r$ increases it become difficult to determine  partial tipping. Therefore, we increase $M$ gradually from $200$ when $r\approx 0.06$ to $20000$ when $r\geq0.24$ to compute the partial tipping region in Figure \ref{fig:paraplane1} effectively. 
\item Considering a large $T>0$, we require $s\leq\Lambda(t)\leq (\lambda_{\max}-s)$ for $t\in[-T,T]$ and some small real number $s$. In our computations we set $s=0.01$ which effectively determines $T$: for the parameter shift $\Lambda(\tau)= \frac{\lambda_{\max}}{2} \left(\tanh\left( \frac{\tau \lambda_{\max} }{2} \right)+1 \right)$, the integration time $T$ can be given as $T=\ln\left(\frac{\lambda_{\max}-s}{s}\right)/(r\lambda_{\max})$ (note however that this will be inadequate near the bifurcations $a=0$ and $0.25$, as noted in the text).
\item We determine which of the $M$ trajectories approach $\Gamma_+^u$ by measuring the distance between the end-point of each trajectory and the equilibrium point $Z_+$. 
\item The stable manifold of $Z_+$, $W^s(Z_+)$, can be computed as initial value problem of the time reversed system \eqref{eq:zLambda} with initial condition $(\lambda_{\max},0,\lambda_{\max}-s)$. 
\item The regions of tracking, partial tipping, and total tipping, and whether $W^s(Z_+)$ limits to $Z_-$ or $\Gamma_-^u$ in the past, are used to characterize six different regions where the behaviour of the system is qualitatively different. These regions are shown in Figure~\ref{fig:paraplane1} and the behaviour of the system at each of them is illustrated in Figure~\ref{fig:manifsect}
\end{itemize}


\bibliography{../bibliography}

\end{document}